[1994/12/01]
\documentclass[final,leqno]{siamltex}
\usepackage[english, activeacute]{babel}
\usepackage{amsmath,amsxtra}
\usepackage{epsfig}
\usepackage{amssymb}
\usepackage{latexsym}
\usepackage{amsfonts}
\usepackage{hyperref}
\usepackage{pdftricks}
\title{On approximate controllability of generalized KdV solitons}
\author{Claudio Mu\~noz\thanks{Laboratoire de Math\'ematiques d'{}Orsay, B\^at. 425, Universit\'e Paris-Sud  91405 Orsay Cedex
France, email: \texttt{claudio.munoz@math.u-psud.fr}}}
\begin{document}
\maketitle

\begin{abstract}
We consider the approximate control of solitons in generalized Korteweg-de Vries equations. By introducing a suitable internal bilinear control on the equation, we prove that any soliton is approximate null controllable, and moreover, any soliton can be accelerated to any particular positive velocity, after a suitable large amount of time. Precise estimates on the error terms and the rate of decay in the approximate null controllability result are also given. Our method introduces a new insight on the control of nonlinear objects, from the point of view of interaction and collision problems for nonlinear dispersive equations, recently developed by Y. Martel and F. Merle \cite{MMcol1,MMcol3}. It can be applied in principle, to several other models with soliton solutions.
\end{abstract}

\date{\today}

\pagestyle{myheadings}
\thispagestyle{plain}
 \markboth{Control of solitons} {Claudio Mu\~noz}


\chardef\bslash=`\\ 
\newcommand{\ntt}{\normalfont\ttfamily}

\newcommand{\cn}[1]{{\protect\ntt\bslash#1}}
\newcommand{\pkg}[1]{{\protect\ntt#1}}
\newcommand{\fn}[1]{{\protect\ntt#1}}
\newcommand{\env}[1]{{\protect\ntt#1}}
\hfuzz1pc 


\newtheorem{thm}{Theorem}[section]
\newtheorem{cor}[thm]{Corollary}
\newtheorem{lem}[thm]{Lemma}
\newtheorem{prop}[thm]{Proposition}
\newtheorem{ax}{Axiom}
\newtheorem{rem}{Remark}[section]
\newtheorem{Cl}{Claim}

\numberwithin{equation}{section}

\newcommand{\secref}[1]{\S\ref{#1}}
\newcommand{\lemref}[1]{Lemma~\ref{#1}}



%

\newcommand{\R}{\mathbb{R}}
\newcommand{\N}{\mathbb{N}}
\newcommand{\Z}{\mathbb{Z}}
\newcommand{\T}{\mathbb{T}}
\newcommand{\Q}{\mathbb{Q}}
\newcommand{\Com}{\mathbb{C}}
\newcommand{\la}{\lambda}
\newcommand{\pd}{\partial}
\newcommand{\wqs}{{Q}_c}
\newcommand{\ys}{y_c}
\newcommand{\al}{\alpha}
\newcommand{\bt}{\beta}
\newcommand{\ga}{\gamma}
\newcommand{\de}{\delta}
\newcommand{\te}{\theta}
\newcommand{\lss}{\lesssim}
\newcommand{\gss}{\grtsim}
\newcommand{\arctanh}{\operatorname{arctanh}}
\newcommand{\spawn}{\operatorname{span}}
\newcommand{\sech}{\operatorname{sech}}
\newcommand{\dist}{\operatorname{dist}}
\newcommand{\re}{\operatorname{Re}}
\newcommand{\ima}{\operatorname{Im}}
\newcommand{\diam}{\operatorname{diam}}
\newcommand{\pv}{\operatorname{pv}}
\newcommand{\sgn}{\operatorname{sgn}}
\newcommand{\vv}[1]{\partial_x^{-1}\partial_y{#1}}

\newcommand{\Lp}{\mathcal{L}_{+,\ve}}
\newcommand{\Lm}{\mathcal{L}_{-,\ve}}
\def\bm{\left( \begin{array}{cc}}
\def\endm{\end{array}\right)}
\def\YY{{\mathcal Y}}

 \providecommand{\abs}[1]{\lvert#1 \rvert}
 \providecommand{\norm}[1]{\lVert#1 \rVert}
 \newcommand{\sublim}{\operatornamewithlimits{\longrightarrow}}
 \newcommand{\ex}{{\bf Example}:\ }
\newcommand{\ve}{\varepsilon}
\newcommand{\fin}{\hfill$\blacksquare$\vspace{1ex}}

\newcommand{\be}{\begin{equation}}
\newcommand{\ee}{\end{equation}}
\newcommand{\ba}{\begin{equation*}}
\newcommand{\ea}{\begin{equation*}}
\newcommand{\bea}{\begin{eqnarray}}
\newcommand{\eea}{\end{eqnarray}}
\newcommand{\bee}{\begin{eqnarray*}}
\newcommand{\eee}{\end{eqnarray*}}
\newcommand{\ben}{\begin{enumerate}}
\newcommand{\een}{\end{enumerate}}
\newcommand{\nonu}{\nonumber}
\newcommand{\ds}{\displaystyle}


\newcommand{\interval}[1]{\mathinner{#1}}

\newcommand{\eval}[2][\right]{\relax
  \ifx#1\right\relax \left.\fi#2#1\rvert}

\newcommand{\envert}[1]{\left\lvert#1\right\rvert}
\let\abs=\envert

\newcommand{\enVert}[1]{\left\lVert#1\right\rVert}
\let\norm=\enVert

\renewcommand{\sectionmark}[1]{}

\section{Introduction}

\medskip

In this paper we consider the problem of controlling solitons of subcritical, generalized Korteweg-de Vries equations (gKdV). More precisely, we look for an \emph{internal control} $f=f(t,x)$ applied to modify the dynamics of generalized \emph{solitons} of the equation
\be\label{gKdV}
u_{t}  +  (u_{xx} + u^p)_x = f, \quad p=2,3 \hbox{ or } 4.
\ee
Here $u=u(t,x)$ is a real-valued function, and $(t,x)\in \R^2$.  When $p=2$ and $f\equiv 0$ (\ref{gKdV}) is the well-known integrable Korteweg-de Vries equation (KdV).

\medskip

Additionally, if $f\equiv 0$, equation (\ref{gKdV}) is a standard gKdV equation. It has special solitary wave solutions called \emph{solitons},\footnote{Strictly speaking, we should say \emph{solitary waves} instead of solitons, but since we are not dealing with integrability issues, we will adopt the denomination soliton. This misunderstanding has spread out through the dispersive models community.} of the form
\be\label{Sol}
u(t,x) = Q_c (x-ct), \quad Q_c(s) := c^{\frac 1{p-1}} Q(\sqrt{c} s), \quad c>0,
\ee
with
\be\label{Q}
Q (s):= \Bigg[\frac{p+1}{2\cosh^2 (\frac 12(p-1)s)} \Bigg]^{\frac1{p-1}}.
\ee
The parameter $c>0$ is usually denoted as the \emph{scaling}, or in a equivalent way, as the \emph{velocity} of the soliton. Inserting the previous profile in (\ref{gKdV}) (recall that $f\equiv 0$), one has that $Q_c>0$ satisfies the nonlinear ODE
\be\label{ecQc}
Q_c'' -c\, Q_c +Q_c^p=0, \quad Q_c\in H^1(\R).
\ee
Moreover, standard conservation laws for (\ref{gKdV}) at the $H^1$-level are the \emph{mass}
\be\label{M}
M[u](t)  :=  \frac 12 \int_\R u^2(t,x)dx = M[u](0),
\ee 
and \emph{energy} 
\be\label{E}
E[u](t)  :=  \frac 12 \int_\R u_x^2(t,x)dx -\frac 1{p+1} \int_\R u^{p+1}(t,x)dx = E[u](0).
\ee 
A satisfactory Cauchy theory is also present at the $H^1$ level of regularity, see e.g. Kenig-Ponce-Vega \cite{KPV}. The condition on $p$ is necessary to get global existence for general $H^1$-data, cf. the paper by Martel and Merle \cite{MMblow} for the critical case $p=5$. When $p>5$, solitons are unstable \cite{BSS}.

\medskip

The control problem for the non inviscid KdV equation in a \emph{finite} length interval has been extensively studied in the last twenty years, starting from the works of Zhang \cite{BYZ}, Russell and Zhang in \cite{RZ,RZ1}, and \cite{RZ2} for a system with periodic boundary conditions and with an internal control. For the case of a boundary control, see \cite{RZ2} and  \cite{Sun}. Concerning the non periodic framework, Rosier studied \cite{R} the controllability of the KdV equation posed on a finite interval of $(0,L)$, under homogeneous Dirichlet boundary conditions and a control acting on the Neumann data at the right end-point of the interval. In particular, Rosier showed that if the length $L$ does not belong to a set of critical values, both the associated linear and the nonlinear systems are exactly controllable. When $L$ is critical, the linear system is not controllable because of the existence of a finite-dimensional subspace of unreachable states. In this case, the exact controllability of the KdV equation, in the case of critical domains, has been proven by Coron-Cr\'epeau \cite{CC}, Cerpa \cite{C}, and Cerpa-Cr\'epeau \cite{C1}. Concerning the exact boundary control problem in the half-line, see e.g. the work of Rosier \cite{R1}.

\medskip

In this paper, unlike the previous results, we are interested in the study of a control problem associated to a given gKdV soliton posed on the real line. The main motivation of our problem will come from the fact that usual techniques from control theory cannot handle some controllability problems posed in unbounded domains, and even worse, the emergency of very particular \emph{nonlinear} solutions cannot be treated using just linear techniques.  

\medskip

Let us explain in more detail  the problem. Given an initial datum $u(t_0,x) =u_0(x)= Q_c(x-ct_0-x_0)$ of soliton type, our objective is to introduce a control $f$ in the gKdV equation (\ref{gKdV}), during an interval of time $[0,T]$, with the purpose of accelerating the soliton to a new soliton state, with a different (positive) velocity. With no loss of generality,  we can assume $t_0=x_0=0$ and that the initial velocity satisfies $c=1$. In other words, our goal is to determine sufficient conditions on $f$ to ensure that, given any final scaling $c_f>0$, the system (\ref{gKdV}) with initial datum $u_0$ evolves to a soliton of the form $Q_{c_f}$, up to some small error terms, in a suitable time of interaction $T>0$. Moreover, we also want to estimate the position of the soliton, compared with the theoretically expected position $\sim c_f T$. 

\medskip

We will assume that the interior control $f$ is given by the  bilinear control (or feedback law)
\[
f(t,x) = a(t,x) u(t,x),
\]
with  $a$ an internal potential satisfying the a priori assumptions
\be\label{a}
a(t,\cdot) \in  C^3(\R) \cap L^2(\R) \cap L^\infty(\R).
\ee 
In other words, our control will use some explicit information of the soliton solution at each time, such as the \emph{scaling} and \emph{position} parameters. This problem has been also considered in a more physical context by Kaup-Newell \cite{KN1}, Grimshaw \cite{Gr1}, Ko-Kuehl \cite{KK}, and Lochak \cite{Lo}.

\medskip

Therefore, in what follows, we consider the initial value problem
\be\label{CP}
\begin{cases}
 u_t + (u_{xx} + u^p)_x =a(t,x)u \; \hbox{ in }\, \R, \quad p=2,3,4, \\
 u(0,x) = Q(x),
\end{cases}
\ee
 where $a$ is an unknown control. Our first result states that any gKdV soliton is approximate null controllability for sufficiently large time.

\medskip

\begin{thm}\label{Cor1}
Any gKdV soliton is approximate null-controllable in large time. More precisely, fix $\delta_0>0$ small. There is $\delta_1=\delta_1(\delta_0)>0$ small such that for all $0<\delta<\delta_1$, the following holds. There exist a time $T=T_\delta>0$, and a smooth control $a=a_\delta(t,Ê\cdot) \in L^\infty(\R) \cap L^2(\R)$, defined in $[0,T]$,  such that the unique solution $u=u_\delta(t)$ of (\ref{gKdV}) in $C([0,T], H^1(\R))$, with initial condition $u_0(x) =Q(x)$, satisfies
\[
\|u(T)\|_{H^1(\R)} \leq \delta.
\] 
Finally, one has $T\sim \delta^{-2(1+\delta_0)}$ and $\sup_t \|a_\delta(t)\|_{L^2\cap L^\infty} \lesssim \delta^2$.
\end{thm}

\medskip

As far as we know, this is the first (partial) result on controllability of \emph{solitons} in unbounded domains, where dispersion plays a key role in the dynamics. Previous results are related to the study of the \emph{ground state} of the linear and the nonlinear problem in a finite interval, see e.g. the works of Lange and Teismann \cite{LT}, Beauchard and Mirrahimi \cite{BM}, Mirrahimi \cite{Mi}, or the control of a quantum particle under the action of a well shaped potential, obeying the linear Schr\"odinger equation in a bounded interval (Beauchard-Coron \cite{BC}), and Cr\'epeau \cite{Cre} in the KdV case. In this paper we study a nonlinear object instead; localized solitons are present due the non compact character of the domain, and the strength of the nonlinearity.

\medskip

It turns out that Theorem \ref{Cor1} is consequence of the following deeper result, a large time \emph{approximate controllability} of the initial soliton $Q$ of scaling one, to any final scaling $c_f>0$, $c_f\neq 1$ (the case $c_f=1$ is trivial). As previously stated, we pick {\bf any $\delta_0>0$ small, but fixed}.

\medskip

\begin{thm}\label{MT} Let $c_f>0$, $c_f\neq 1$. There exists $\ve_0(\delta_0,c_f)>0$ such that, for all $0<\ve<\ve_0$ the following holds. There exist a time $T=T_\ve>0$, a smooth in time and space control $a=a_\ve(t,Ê\cdot) \in L^\infty(\R) \cap L^2(\R)$ and a smooth translation parameter $\rho(t)$, both defined in $[0,T]$, and such that the unique solution $u=u_\ve(t)$ of (\ref{gKdV}) in $C([0,T], H^1(\R))$, with initial condition $u_0(x) =Q(x)$, satisfies  
\be\label{MT1}
\|u(T)- Q_{c_f}(\cdot - \rho(T)) \|_{H^1(\R)} + |\rho'(T)- c_f | \lesssim \sqrt{\ve}.
\ee
Finally, one has $T\sim \ve^{-1-\delta_0}$ and $\sup_t \|a_\ve(t)\|_{L^2\cap L^\infty} \lesssim \ve$.
\end{thm}

%

\medskip

\begin{proof}[Proof of Theorem \ref{Cor1}, assuming Theorem \ref{MT}] The proof of Theorem \ref{Cor1} follows from Theorem \ref{MT}, and the subcritical character of solitons for $p<5$. Indeed, just take any $0<c_f \leq \frac 1{100}\delta^{\frac{4(p-1)}{5-p}}$ and $\ve_0 \lss \delta^2$ in Theorem \ref{MT}. From \eqref{Sol} we have, after integration and rescaling, 
\[
\|Q_{c_f}\|_{H^1(\R)} \sim c_f^{\frac{5-p}{4(p-1)}} < \frac12 \delta.
\]
Therefore, using \eqref{MT1},
\[
\|u(T) \|_{H^1(\R)}  \lesssim \sqrt{\ve} + \| Q_{c_f}(\cdot - \rho(T))\|_{H^1(\R)} \leq \frac 12\delta + \sqrt{\ve} \leq \delta.
\]
Note that Theorem \ref{Cor1} holds  even \emph{without destroying the soliton structure}.
\end{proof}

Some comments about Theorem \ref{MT}. 

\medskip

\begin{rem}
First of all, we point out that the control $a$ is not compactly supported, but it satisfies the following properties (see Section \ref{2} and (\ref{aexp}) for more details):
\ben
\item It is exponentially decreasing in any moving region far away form the soliton (in other words, it moves with the soliton);
\item It has slow variation in space, which actually explains the large time needed in order to drive the dynamics. 
\een
\end{rem}

\medskip

\begin{rem}
Second, as for the final position and time of control $T$ are concerned, we obtain  estimates of the following orders: for any $\delta_0>0$ small but fixed,
\[
T\sim \ve^{-1-\delta_0}, \qquad |\rho(T)- c_f T| \lesssim \ve^{-1/2-\delta_0},
\]
although the \emph{relative} error in the last estimate is $O(\ve^{1/2})$. The relative weakness in $\ve$ of the last estimates and the bound (\ref{MT1}) is mainly due to the emergence of \emph{dispersive tails} behind the soliton solution as the control acts; this phenomenon has been observed in several interaction problems involving gKdV equations, starting from the formal arguments in \cite{KN1,KK,Gr1}, and the more rigorous treatment given in \cite{MMcol1,MMcol3,Mu2,Mu3,H,HPZ}. This phenomenon does not appear in the case of nonlinear Schr\"odinger equations, where one expects better estimates (see e.g. \cite{Mu3}). Heuristically speaking, the lack of control on the position, compared to the extremely accurate control on the velocity, could be associated to a form of \emph{uncertainty principle} for solitons, regarded this time as almost point particles.
\end{rem}

\medskip

\begin{rem}
A necessary condition to obtain an estimate as in (\ref{MT1}) is the lack of conserved quantities (see Proposition \ref{Cau}). Indeed, it is not difficult to see that the soliton at time $T$ has lost or gained, depending on the sign of $a$ and $c_f$, a nontrivial $O(1)$ amount of mass (\ref{M}). A similar study can be applied to the case of the energy (\ref{E}), with similar conclusions. It is important to stress that, since solitons are stable under small $H^1$ perturbations \cite{Benj,BSS,We1} and the equation is not integrable unless $p=2$ or $p=3$, it is expected that the result above only holds if we introduce a sufficiently slowly varying potential.
\end{rem}

\medskip

On the other hand, problem (\ref{gKdV}) can be also regarded as a stabilization problem. In that sense, the recent literature concerns with the decay of solutions posed in a bounded interval \cite{RZ0,MMP,BZ,LRZ}, or the half line by Linares and Pazoto \cite{LP0,LP,PR}, and numerical schemes for the critical case $p=5$ (Pazoto et. al. \cite{PSV}). As for the decreasing mass case, and the approximate null controllability result stated in Theorem \ref{Cor1}, we have the following additional approximate stabilization result, without destroying the soliton:

\medskip

\begin{cor}\label{Cor2}
Under the assumptions of Theorem \ref{Cor1}, there exist $C,\mu_0>0$, independent of $\delta \in (0,\delta_1)$ such that, for all $t\in[0,T]$
\be\label{Stab}
\|u(t)\|_{H^1(\R)} \leq C (\delta + e^{-\mu_0 \delta^2 t}) \|Q\|_{H^1(\R)}, \qquad p=2,3,4.
\ee
\end{cor}
\medskip

\begin{rem}
Finally, some words about the corresponding exact controllability problem. A nice exact controllability result could be obtained if we were able to prove e.g. exact null controllability of small solitons, and then combining Theorems \ref{Cor1} and \ref{MT1} in the standard way. However, based on some results about inelasticity of the dynamics for slightly perturbed solitons (cf. \cite{Mu2,Mu3}), we believe that in our model, and in more general situations, solitons are never exactly controllable, even in infinite time. 

Indeed,  note that if \eqref{CP} is exactly controllable to zero in finite time, say $u(T)=0$ for some bounded, smooth control $a(t,x)$, then using the reversibility in time of the equation, we have that $v(t,x):= u(T-t,-x)$ satisfies a slightly different equation,
\[
v_t + (v_{xx} + v^p)_x = b(t,x)v, \quad v(0)=0,
\]
for the potential $b(t,x) :=a(T-t,-x)$. It turns out that, under standard assumptions on the solvability of the Cauchy problem associated to $v$, the unique solution to the above problem is the identically zero solution, a contradiction. I thank Sylvain Ervedoza for this remark.
\end{rem}

\medskip

 \subsection{About the proofs}  Our proofs do not involve the usual methods employed in control theory, requiring e.g. the study of the linear problem, unique continuation properties and/or Carleman estimates. In order to study genuine nonlinear objects such as solitons, we need different dispersive methods. In particular, a suitable global well-posedness theory in the energy space for solutions of (\ref{gKdV}) in the real line requires modifications on the arguments of the fundamental work by Kenig, Ponce and Vega \cite{KPV}, in order to deal with the unbounded domain case. Second, our control is explicitly constructed, with the following properties: $(i)$ it has a slowly varying character, determined by the parameter $\ve$; $(ii)$ it is localized in a moving region of size $O(1)$, and is of strength $O(\ve)$ (but it is not compactly supported), and $(iii)$ the corresponding slowly varying part induces on the soliton parameters a finite dimensional dynamical system which governs the whole dynamics.  Concerning the time of control, since the dynamics is slowly varying, the time of interaction is $O(\ve^{-1})$ at least; a large control introduced in a smaller window of time could destroy the soliton.
 
 \medskip
 
 The second step of the proof is the following: since the introduction of the control induces on the soliton the action of an external potential, we can think such an interaction as a slowly varying collision between both objects. In Section \ref{3}, we construct an explicit approximate solution which describes the interaction, up to certain order of accuracy in $\ve$. This solution $\tilde u(t)$ has the form
\[
u(t,x) \sim Q_{c(t)} (x-\rho(t))  + \ve A(t,x-\rho(t)),
\]
where $\ve>0$ is a small, artificially introduced parameter, and $(c(t),\rho(t))$ are suitable scaling and translation parameters, depending on time. The parameters follow a suitable approximate finite-dimensional, \emph{slowly varying} in time dynamics, determined by the action of the control, described as follows:
\[
c'(t) \sim  a_x(t, \ve \rho(t)), \quad \rho'(t) \sim c(t) +O(\ve),
\]
where $a$ is the control introduced in (\ref{aexp}). We choose carefully $a$ such that the evolution of this system leads to the desired final velocity, at time $T\sim \ve^{-1-\delta_0}$,
\be\label{cpT}
c(T) \sim c_f, \quad \rho(T) \sim c_f T  + o(T),
\ee
however, a better control on the position has escaped to us.

\medskip

Concerning the function $A$, it corresponds to a first order correction term with support  of size $O(\ve^{-1})$ in the variable $x-\rho(t)$ (the soliton variable), and $L^\infty$-norm of order $O(e^{-\ga_0 \ve |\rho(t)| })$, for some constant $\ga_0>0$. Therefore, $A$ is a phantom term that disappears after the interaction, but which allows to improve the accuracy of the approximate solution. \emph{Finding $A$ is an absolutely necessary condition, otherwise a bound like \eqref{MT1} is highly unlikely}. Additionally, $A$ is in principle only bounded, but not localized,\footnote{In principle, $A$ models a \emph{dispersive tail} behind the soliton solution.} therefore we introduce a suitable cut-off function to recover a finite mass solution.
The error associated to this approximation is measured in terms of the $L^\infty_t H^1_x$ norm, and it has to be small enough in order to take into account the large time of interaction. In our case, we are able to prove that during the whole interaction, one has
\[
\hbox{error} \sim \ve^{3/2} e^{-\ga_0 \ve |\rho(t)| }, \quad t\in [0, T],
\]
(see (\ref{SSS})), therefore the propagation of this error during a time interval of order $\sim T$ formally leads to the bound $O(\sqrt{\ve})$ in Theorem \ref{MT}.  We remark that this method has been recently applied, in a different context, to several interaction problem, notably the two-soliton collision by Y. Martel and F. Merle \cite{MMcol1,MMcol3}, and the interaction of solitons with a potential \cite{Mu2,Mu3}. See also \cite{QL} for a related soliton-potential problem in a different context, in the easier case of the cubic nonlinearity, and for which the term $A$ is not needed. 
 
\medskip 
 
The third step of the proof is the following. In order to control the dynamics of the error terms, we introduce a suitable Lyapunov functional (Section \ref{4}), adapted this time to the genuine nonlinear dynamics of the problem  (see e.g. \cite{MMcol1}). This functional has very small variation in time, provided we control the size of some time dependent parameters of the soliton solution. We avoid that problem by using sharp virial estimates, in the spirit of \cite{MMnon}.  After this point, we can close the main argument by proving rigorously that the error terms can be assured to be smaller than $O(\sqrt{\ve})$, during the whole interaction region. 

\medskip

The final step of the proof is a rigorous analysis of the parameters $(c(t),\rho(t))$ of the soliton solution, in order to recover (\ref{cpT}). We prove that at time $t=T$, the solution has the desired behavior, up to an error of $O(\sqrt{\ve})$, finishing the proof of Theorem \ref{MT}. Finally, the proof of Corollary \ref{Cor2} follows after a detailed study of the scaling parameter $c(t)$.

\medskip

The weakness of our approach is precisely the approximate character of the controllability property, and the large time needed to reach an approximate final state. We believe that our results can be improved by adapting to this case, the standard and complex machinery of control theory.  Additionally, we believe that the moving profile of the support can be chosen to be compactly supported.
 
\medskip 

We point out that in order to describe the dynamics in a time of order $O(1)$, one formally needs a large control; in particular, it should be unbounded in space (more precisely, linearly growing in space). However, even the local in time Cauchy problem for such perturbations becomes a very difficult problem.

 \medskip
 
 Finally, some words about the organization of this paper. In Section \ref{2}, we introduce the explicit control system, the finite dimensional dynamical system and the corresponding local and global well-posedness theory. In Section \ref{3} we construct an approximate solution to a given order of accuracy. We continue this process up to the moment when we find an infinite mass correction term, which is up to date the best mathematical description of the dispersive tail behind the soliton, originated by the application of the control. Section \ref{4} is devoted to the introduction of a Lyapunov function, modulation theory and a key virial identity in order to control the dynamics of the oscillatory terms. Finally, in Section \ref{5} we prove the main theorem.
 
\medskip 
 
{\bf Notation.} Along this paper we use the convention $A\lss B$ if and only if there exists $K>0$, independent of $\ve$, such that  $A\leq KB$. Additionally, $\ga$ and $K^*$ will denote special positive constants, still independent of $\ve$, to be worried about. Finally, $\mathcal S(\R)$ denotes the Schwartz's class on $\R$.
 
\medskip

{\bf Acknowledgments.}  I would like to thank the referees for their useful and constructive critiscisms. I also thank Gunther Uhlmann and Axel Osses for their kind invitation to the PASI-CIPPDE 2012, Inverse Problems and PDE Control, held in Santiago-Chile, and where this project was originally conceived. Finally I'm grateful of Eduardo Cerpa and Sylvain Ervedoza, for many useful comments and suggestions to a first draft of this paper.

\bigskip

\section{First ingredients}\label{2} 
 
\medskip

Let  $p=2,3$ or $4$.  Given any $c_f>0$ fixed and $Q$ be the soliton defined in (\ref{Q}), we define the quantities
\be\label{lap}
\la_p:= \frac{4(p-1)}{5-p} \frac{\int_\R Q^3}{\int_\R Q^2}>0, 
\ee
\be\label{ainf}
a_\infty:= -\frac 1{\la_2} \log c_f , \quad p=2, \qquad a_\infty:= \frac{(p-1)}{\la_p(p-2)}(1-c_f^{\frac{p-2}{p-1}} ), \quad p=3,4.
\ee
Note that, as expected, for every $p$ the value of $a_\infty(c_f)$ tends to zero as $c_f$ approaches the trivial case $c_f=1$ (i.e. no control is needed). 

\medskip

We introduce now the control $a(t,x)$. Given any $\ve>0$ small, we consider a smooth function $a_0$ satisfying the  following properties (recall that $A\lesssim B$ means that there is $C>0$ such that $A \leq CB$)
\bea\label{ahyp}
\begin{cases}
 a_0\in C^3(\R) \cap L^\infty(\R) , \\
 |a_0(x)| \lesssim e^{\ga_0 x}, \hbox{ for } x\leq -1,  \quad | a_\infty-a_0(x)| \lesssim e^{-\ga_0 x}   \hbox{ for } x\geq 1,  \\
 |a_0^{(k)}(x)| \lesssim e^{-\ga_0 |x|}, \quad x\in \R, \; k=1,2,3, \\
 a_0'(x)>0 \; \hbox{ if } a_\infty>0, \quad a_0'(x)<0 \;\hbox{ if } a_\infty<0,
\end{cases}
\eea
for a fixed, positive constant $\ga_0.$  Note that with this choice,
\be\label{Linf}
\|a_0\|_\infty =|a_\infty|.
\ee

Let $0<c_m:= \frac 12 \min \{ c_f,1\}$, $c_M:= 2\max \{ 1, c_f\}$, and $(c_0(t), \rho_0(t)) \in \R_+\times \R$ be a set of  $C^1$ parameters defined in $\{ t\geq 0\}$, with the following uniform, a-priori constraints
\be\label{r1}
0<c_m  \leq c_0(t) ,\, \rho_0'(t)  \leq c_M.
\ee
More precisely, consider $t\geq 0$ and $(c_0(t),\rho_0(t)) \in \R_+ \times \R$ be the unique solution of the nonlinear ODE system
\be\label{ODE}
\begin{cases}
c_0'(t) =\ve f_1^0(c_0(t),\rho_0(t)), \qquad  c_0(0)=1,\\
\rho_0'(t) = c_0(t),\qquad \rho_0(0)=-\ve^{-1-\delta_0}, 
\end{cases}
\ee
where $f_1^0$ is defined as follows (cf. (\ref{lap}))
\be\label{f10}
f_1^0(c,\rho) := -\la_p a_0'(\ve \rho) c^{\frac p{p-1}},
\ee
and $\delta_0>0$ is the small parameter of Theorems \ref{Cor1} and \ref{MT}. Additionally, we will need the following function
\be\label{f20}
 f_2^0(c,\rho) := \mu_p  a_0'(\ve \rho)c^{\frac{2(5-2p)}{7-3p}},
\ee
for some $\mu_p\in \R$, with $\mu_3 =0$ (note that $\frac{2(5-2p)}{7-3p}>0$ for $p=2,3,4$). 
\medskip

\begin{lem}\label{ODE0}
There exists a unique solution $(c_0(t),\rho_0(t)) \in \R_+ \times \R$ of (\ref{ODE}), defined for all $t\geq 0$. Moreover, we have
\be\label{ODE1}
\lim_{t\to +\infty} c_0(t) = c_f(1+O(\ve^{10})), \qquad \lim_{t\to +\infty} \rho_0(t) = +\infty,
\ee
and
\be\label{ODE2}
c_m  \leq c_0(t)  \leq c_M,
\ee
for $\ve$ small enough.
\end{lem}
\begin{proof}
The existence of a unique local solution to (\ref{ODE}) is a direct consequence of the Cauchy-Lipschitz-Picard theorem. The global character of the solution is directly determined by the boundedness of $a_0$. 

\medskip

Let us prove (\ref{ODE1}) and (\ref{ODE2}). First of all, note that $(c_0(t),\rho_0(t))\equiv (0, \hbox{constant})$ is a constant solution of (\ref{ODE}), without considering the initial conditions. Therefore, we have $c_0(t)>0$ for all $t\geq 0$. On the other hand, from the first equation in (\ref{ODE}),
\bee
c_0^{-\frac1{p-1}}(t) c_0'(t) & =&  -\ve \la_p a_0'(\ve \rho_0(t)) c_0(t)   = -\ve \la_p a_0'(\ve \rho_0(t)) \rho_0'(t) .
\eee
Hence, if $p=2$,
\[
\log c_0(t)  = -\la_2 [a_0(\ve \rho_0(t)) - a_0(-\ve^{-\delta_0})].
\]
from which we obtain for $\ve$ small, using (\ref{ahyp}),
\be\label{C20}
c_0(t) = e^{-\la_2 a_0(\ve \rho_0(t))}(1+O(\ve^{10})), \quad p=2,
\ee
with the term $O(\ve^{10})$ independent of time. Similarly, if $p=3$ or $4$,
\be\label{Cp0}
c_0(t) = \Big[  1-  \la_p \frac{(p-2)}{p-1} a_0(\ve \rho_0(t)) \Big]^{\frac{p-1}{p-2}} (1+O(\ve^{10})).
\ee
Note that from (\ref{C20})-(\ref{Cp0}), (\ref{ainf}) and (\ref{ahyp}), $c_0(t)$ satisfies the bounds
\[
0< c_m = \frac 12 \min\{ c_f,1\} \leq c_0(t) \leq 2 \max \{ c_f,1\} =c_M .
\]
This shows (\ref{ODE2}). We conclude that for $\ve>0$ small, $\rho_0(t)$ is increasing and $\rho_0(t) -\rho_0(0) \geq c_m t $,
which implies that $ \lim_{\infty}\rho_0=+\infty$. Moreover, from (\ref{ahyp}), 
\[
 \lim_{+\infty}a(\ve \rho_0(t)) = a_\infty .
 \]
 Therefore
\[
\lim_{+\infty}c_0(t) = e^{ -\la_2 a_\infty} (1+O(\ve^{10})) = c_f (1+O(\ve^{10})) , \quad p=2,
\]
and similarly for $p=3, 4$. This proves (\ref{ODE1}).
\end{proof}

\medskip

Finally, define
\be\label{aexp}
a(t,x):= -\ve a_0'(\ve x)Q_{c_0(t)}(x-\rho_0(t)),
\ee
where $Q_c$ is the solution of (\ref{ecQc}) and $(c_0(t),\rho_0(t))$ is the solution of (\ref{ODE}).
Let us remark that this control takes into account important information of the soliton itself, namely the approximate scaling $c_0(t)$ and position $\rho_0(t)$, and it is in some sense of nonlinear character. In terms of numerical applications, these two parameters can be easily described by solving the ODE (\ref{ODE}). The non stationary character of this control will become essential in the proof. 

\medskip

It is not difficult to check that this control satisfies the following space-time bounds
\[
\|a\|_{L^\infty([0,\infty)\times \R)} + \|a_{xx}\|_{L^\infty([0,\infty)\times \R)} \lesssim  \ve.
\]
Under these estimates, we claim that the Cauchy problem associated to (\ref{gKdV}) is locally well-posed in a subspace of $H^1(\R)$.
\medskip

\begin{prop}\label{Cau}
Under the assumptions (\ref{ahyp}), (\ref{r1}), (\ref{ODE}) and (\ref{aexp}), the initial value problem (\ref{CP}) is locally well-posed in $H^1(\R)$. Moreover, the mass $M[u](t)$ and energy $E[u](t)$ defined in (\ref{M}) and (\ref{E}) satisfy the relations
\be\label{dME}
\partial_t M[u](t) =  \int_\R a(t,x) u^2, \qquad \partial_t E[u](t) =- \frac 12 \int_\R a_{xx}u^2 - \int_\R a u^{p+1} +\int_\R a u_x^2. 
\ee
\end{prop}

\begin{rem}
Later we will prove that our solution is well-defined, for all $t\leq T\sim \ve^{-1-\delta_0}$, as a consequence of the stability property (\ref{MT1}).
\end{rem}

\medskip

\begin{proof}
This result is classical, see e.g. Merle-Vega \cite{MV} in the case where $p=3$ (the so called mKdV equation) and the nonlinearity $u^p$ has the opposite sign. For the sake of completeness, we sketch the main details. We use the machinery developed by Kenig, Ponce and Vega \cite{KPV} to prove local well-posedness for gKdV in low regularity Sobolev spaces. Since we only need an $H^1$ local theory, our proof will be simpler than the original one.

\medskip

Recall that we want to solve 
\[
u_t + u_{xxx} = - [p u^{p-1}u_x + a(t,x)u], \qquad u(0) = u_0 \in H^1(\R).   
\]
If we denote by $e^{-t\partial_x^3}$ the free Airy propagator, we have to solve fixed point problem
\bee
u(t)  & = &  \mathcal T[u](t) := e^{-t\partial_x^3}u_0 - \int_0^t  e^{-(t-s)\partial_x^3}[p u^{p-1}u_x + a u](s)ds\\
&  := & \mathcal T_0[u_0] (t)+ \mathcal T_1[u](t).
\eee
Note that, since we have chosen $(c,\rho)$ following (\ref{r1}), 
\[
\int_\R \sup_{ 0\leq t\leq 1} a^2 \lss \ve^2 \int_\R a_0'^2(\ve x) dx \lss \ve, \qquad \sup_{0\leq t \leq 1} \|a\|_{L_x^\infty(\R)} \lss \ve.
\]
Using \cite[Theorem 3.5]{KPV} and the maximal function estimate \cite[(3.9)]{KPV}, we have, for any $S\in (0,1)$,
\bee
\sup_{0\leq t \leq S}\| \partial_x \mathcal T_1[u]\|_{L^2(\R)} & \leq&  \int_{\R_x} \| p u^{p-1}u_x + a u\|_{L^2(0\leq t \leq S)} dx \\
& \lss & \Big( \int_\R \sup_{0\leq t \leq S} |u|^{2p-2} dx\Big)^{1/2} \|u_x\|_{L^2([0,S] \times \R)}  \\
& & \qquad +\Big( \int_\R \sup_{0\leq t\leq S} a^2 dx\Big)^{1/2} \|u\|_{L^2([0,S] \times\R)} \\
& \lss &   S^{1/2}\|u\|^{p}_{L^\infty((0,S); H^1(\R))}  +  \sqrt{\ve} S^{1/2}\|u\|_{L^\infty((0,S); L^2(\R))}.
\eee
On the other hand,
\bee
\|\mathcal T_1[u]\|_{L^2(\R_x)} & \lss&  \int_0^S [\| u^{p-1}u_x \|_{L^2(\R_x)}+ \| a u\|_{L^2(\R_x)} ]dt \\
& \lss & S  \sup_{0\leq t\leq S}\Big( \int_0^S \|u\|^{p-1}_{L^\infty(\R_x)} \Big)^{1/2} \|u_x\|_{L^2(\R_x)} \\
& &   + S\sup_{0\leq t\leq S}  \|a\|_{L^\infty(\R_x)} \|u\|_{L^2(\R)}.
\eee
It is not difficult to check that these estimates give that, for $S$ small, $\mathcal T$  maps a ball of $H^1$ into itself.
The contraction follows in a similar way.

\medskip

Let $T_0>0$ be the maximal time of existence of a solution $u(t)$. It is not difficult to check that the mass and energy (\ref{M})-(\ref{E}) satisfy, for all $t\in [0, T_0)$,
\[
\partial_t M[u](t) =  \int_\R a(t,x) u^2 \lesssim  \ve M[u](t),
\]
therefore $M[u](t) \lesssim e^{C\ve t}$. On the other hand, the energy (\ref{E}) satisfies the relation
\[
\partial_t E[u](t)  = - \frac 12 \int_\R a_{xx}(t,x)u^2 - \int_\R a(t,x) u^{p+1} +\int_\R a(t,x) u_x^2.
\]
\end{proof}

\bigskip

\section{Approximate solution}\label{3}

\medskip

Given any $\ve>0$ and $\delta_0>0$ small, we introduce the time of interaction
\be\label{T}
T := \min\{ T_0, \ve^{-1-\delta_0}\},
\ee
where $T_0>0$ is the maximal time of existence of the solution $u(t)$ with initial condition $Q(x)$.

\medskip

Let $t\in [0, T]$. In what follows, we \emph{fix} a couple of dynamical parameters $(c(t), \rho(t))$, a perturbation of the couple $(c_0(t),\rho_0(t))$, and satisfying  the same estimates (\ref{r1}) in the same subinterval of $\{t\geq 0\}$. Additionally, we will assume that 
\be\label{r2}
|c(t) -c_0(t)| +|\rho(t) -\rho_0(t)| \leq \ve^{1/2-\delta_0}, 
\ee
for all $t\in [0,T].$ 

\medskip

Now we define the modulated soliton solution as follows. Let
\[
y:= x-\rho(t), \qquad  R(t,x) := Q_{c(t)} (y). 
\]
Finally, we introduce the approximate solution
\be\label{defW}
\tilde u(t,x) := R(t,x) + w(t,x), \quad w(t,x):= \ve d(t) A_{c(t)}(t,y), \quad d(t):= a_0'(\ve \rho(t));
\ee
for some $L^\infty(\R)$ function $A_c(t,\cdot)$, to be introduced later. In order to simplify some computations, we will assume that for $c_m\leq c(t)\leq c_M$ and $t$ fixed, $A_{c(t)}(t,\cdot)$ satisfies the estimates
\be\label{Ac}
A_c'(t,\cdot ) \in L^2(\R), \qquad \partial_c A_c(t,\cdot) \in L^\infty(\R),
\ee
that will be verified below. Finally, we define the scaling operator
\be\label{LaQc}
\Lambda Q_c(y) := \partial_{c'} {Q_{c'}}\big|_{ c'=c} (y)= \frac 1c \Big[\frac 1{p-1} Q_c(y) + \frac 12 yQ'_c(y) \Big] \in \mathcal S(\R).
\ee

\medskip

We want to measure the size of the error induced by inserting $\tilde u$ as defined in (\ref{defW}) in the equation (\ref{CP})-(\ref{aexp}). Let
\be\label{2.2bis}
S[\tilde u](t,x) := \tilde u_t + (\tilde u_{xx}  + \tilde u^{p})_x + \ve a_0'(\ve x) Q_{c_0(t)}(x-\rho_0(t))\tilde u .
\ee
From (\ref{r2}) we have
\bea\label{Estimate}
 Q_{c_0(t)}(x-\rho_0(t)) &=& Q_{c(t)}(x-\rho(t)) + \Lambda Q_{c(t)}(x-\rho(t))(c_0(t) -c(t))\nonu  \\
 & &   + Q_{c(t)}'(x-\rho(t))(\rho_0(t) -\rho(t))   + O_{H^1(\R)}(\ve^{1-2\delta_0}).
\eea
Our first result is the following

\medskip

\begin{prop}\label{prop:decomp} Let $(c,\rho)$ be satisfying (\ref{r1}) and (\ref{r2}). There exists a function $A_c\in L^\infty(\R)$ such that $\tilde u$, defined in (\ref{defW}), satisfies\footnote{The first two terms in (\ref{Sdeco}) are often referred as the finite-dimensional \emph{dynamical system} associated to the soliton dynamics.}
\be\label{Sdeco}
S[\tilde u](t,x) = (c'(t)- \ve f_1(t)) \partial_c \tilde u - (\rho'(t) -c(t) -\ve  f_2(t))\partial_y\tilde u + \tilde S[\tilde u],
\ee
where $ f_1(t) =  f_1(c(t),\rho(t))$ and $ f_2(t) =  f_2(c(t),\rho(t))$ are given by
\bea\label{f1}
 f_1(c(t),\rho(t)) & =&  f_1^0(c(t),\rho(t)) + O(a_0'(\ve\rho(t))|c_0(t)-c(t)| ), \\
 f_2(c(t),\rho(t)) &=& f_2^0(c(t),\rho(t)) \label{f2}  \\
 & & + O(a_0'(\ve\rho(t))(|\rho_0(t)-\rho(t)| +|c_0(t)-c(t)| )), \nonu
\eea
with $ f_1^0(t), f_2^0(t)$ defined in (\ref{f10})-(\ref{f20}). Moreover,
\be\label{Rig}
\| \tilde S[\tilde u](t)\|_{H^1(y>-\frac 2\ve)} \lss \ve^{3/2}e^{-\ga \ve |\rho(t)|} + \ve |c(t)-c_0(t)|  e^{-\ve\ga|\rho(t)|} +\ve^3,
\ee
and
\be\label{SIn0}
\abs{\int_\R Q_c \tilde S[\tilde u]} +\abs{\int_\R yQ_c \tilde S[\tilde u]} \lss \ve^2 e^{-\ve\ga|\rho(t)|} +  \ve |c(t)-c_0(t)| e^{-\ve\ga|\rho(t)|}+ \ve^3.
\ee
\end{prop}

\medskip

\begin{proof}
We follow the strategy described in \cite{Mu3}. Suppose that the parameters $(c(t),\rho(t))$ satisfy (\ref{r1}) and (\ref{r2}). From  (\ref{2.2bis}), we have 
\begin{equation}\label{eq:sion}
S[\tilde u]= \bf I +II +III,
\end{equation}
where (we omit the dependence on $t$ and $x$ if it is not necessary) 
\be\label{eq:sion1}
{\bf I} := S[R], \quad {\bf II} = {\bf II}(w) := w_t  +  (w_{xx}  + p R^{p-1} w)_x + \ve a_0'(\ve x) Q_{c_0}(x-\rho_0) w, 
\ee
and for $p=2,3$ or $4$,
\be\label{eq:sion2}
{\bf III} : = \left\{ (R + w)^p - R^p - pR^{p-1}w \right\}_x.
\ee
Recall that $w$ is given by (\ref{defW}). In the next results, we expand the terms in \eqref{eq:sion}. Note that $R(t,x) = Q_{c(t)} (y)$ and $y= x-\rho(t). $

\medskip

\begin{lem}\label{lem:SQ} 
\[
{\bf I}= F_0^{\bf I}(t,y) + \ve  F_1^{\bf I}(t, y)  + \ve^{2(1-\delta_0)}  F_c^{\bf I}(t,y),
\]
where 
\[
F_0^{\bf I}(t,y) := ( c'(t) -\ve f_1(t) ) \partial_c R(t) - (\rho'(t) - c(t)-\ve f_2(t))  \partial_x R(t),
\]
$f_1(t)$ and $f_2(t)$ are given by (\ref{f1})-(\ref{f2}), and
\bea\label{F1Q}
F_1^{\bf I}(t; y) & := & f_1(t) \Lambda Q_c(y)  + a_0'(\ve \rho) Q_c^2(y)  - f_2(t) Q_c'(y)  \nonu \\
& & + (c_0-c)a_0'(\ve \rho) \Lambda Q_c Q_c(y) + (\rho_0-\rho)a_0'(\ve \rho) Q_c' Q_c(y).
\eea
Finally, for all $t\in [0, T] $, one has  $\|F_c^{\bf I}(t,\cdot)\|_{H^1(\R)} \lss e^{-\ve\ga|\rho(t)|} +\ve $.
\end{lem}

\medskip

\begin{proof}[Proof of Lemma \ref{lem:SQ}.] We have
\bee
{\bf I}  & =&   R_t + (R_{xx} +  R^p )_x + \ve a_0'(\ve x)Q_{c_0}Q_c \\
& =&   c'\Lambda Q_c - \rho'  Q_c' +  Q_c^{(3)} + (Q_c^p)' +  \ve a_0'(\ve x) Q_{c_0}Q_c.
\eee
On the other hand, note that via a Taylor expansion,
\bee
a_0'(\ve x)Q_c  =  a_0'(\ve \rho) Q_c + \ve a_0''(\ve \rho) yQ_c + O_{H^1(\R)}(\ve^2).
\eee
Therefore, using the equation satisfied by $Q_c$, namely, $Q_c'' -cQ_c +Q_c^p=0$, (\ref{Estimate}) and (\ref{r2}), we have
\bee
{\bf I} &= & c'\Lambda Q_c - (\rho' -c) Q_c' + Q_c^{(3)} -cQ_c' + (Q_c^p)' + \ve a_0' Q_{c_0}Q_c  \\
& &   +  \ve^2 a_0''y Q_{c_0} Q_c + O_{H^1(\R)}(\ve^3) \\
&  = & (c' -\ve f_1) \Lambda Q_c - (\rho' -c -\ve f_2)Q_c'  \\
& &  +  \ve \big[ f_1 \Lambda Q_c  + a_0' Q_c^2  - f_2 Q_c'  +(c_0-c)a_0' Q_c\Lambda Q_c +(\rho_0-\rho)a_0' Q_c'Q_c \big]  \\
& & + \ve^{2(1-\delta_0)} F_c^{\bf I}(t,y),
\eee
with  $F_c^{\bf I}(t,\cdot )\in \mathcal S(\R)$ and $\| F_c^{\bf I}(t,\cdot) \|_{H^1(\R)} \lss e^{-\ve\ga|\rho(t)|} +\ve.$
\end{proof}

Now we recall the linearized elliptic gKdV operator. Fix $c>0$,  $p=2,3$ or 4, and let
\be\label{defLy}
 \mathcal{L} \bar w := - \bar w_{yy} + c \bar w - p Q_c^{p-1}(y) \bar w, \quad\hbox{ where }\quad  Q_c(y) := c^{\frac 1{p-1}} Q(\sqrt{c} y).
\ee
Here $\bar w=\bar w(y)$. 

\medskip

\begin{lem}\label{lem:dSKdVw} Suppose that $A_c$ satisfies (\ref{Ac}). Let $w$ given by (\ref{defW}). Then the following expansion holds:
\bee
{\bf II} & = &  (c' -\ve f_1)\partial_c w - (\rho' -c -\ve f_2) w_y  - (\mathcal{L} w)_y   \\
& & + \ \ve^2 \big[ a_0'' c A_{c} +  f_1 a_0'  \partial_c A_{c}  \big] + \ve^2  F_c^{\bf II}(t; y),
\eee
with 
\bee
F_c^{\bf II}(t; \cdot) & =&  \ve  a_0''(\ve\rho(t)) f_2(t) A_c + a_0''(\ve\rho(t))  (\rho'(t) - c(t) - \ve f_2(t)) A_c \\ 
& & +\ve^{-1} d(t) \partial_t A_c + O_{H^1(\R)} ( e^{-\ve \ga |\rho(t)|}),
\eee
for some fixed $\ga>0$.
\end{lem}

\medskip

\begin{rem}
It turns out that the term $\ve^{-1} d(t) \partial_t A_c$ will be a very problematic term to estimate; for a delicate treatment of this term see (\ref{Marca}).
\end{rem}

\medskip

\begin{proof}
Let $D:=D_c(t,y)$, $y= x-\rho(t)$, be a general, smooth function. We compute
\[
{\bf II}(D) := D_t + (D_{xx}  + p   R^{p-1} D)_x + \ve a_0'(\ve x) Q_{c_0} D.
\]
We have
\bee
{\bf II}(D)& =& c'(t) \partial_c D + D_t  - (\rho'(t)-c(t))D_y \\
& &  +\big[ D_{yy} - c(t) D+ pQ_c^{p-1} D \big]_x + \ve a_0'(\ve x) Q_{c_0} D \\
& = &   D_t  -(\mathcal L D)_y    + (c'(t) - \ve  f_1(t) ) \partial_c D \\
& &   - (\rho'(t) -c(t) -\ve f_2(t) ) D_y  +  O(\ve a_0'(\ve x) Q_{c_0} D) \\
& & + \ve  f_1(t)  \partial_c D - \ve  f_2(t)   D_y.
\eee
We apply this last identity to the function $w=\ve d(t)A_c(t,y)$. We have
\bea
{\bf II}(w) & = &  \ve d'(t) A_c +\ve d(t) \partial_t A_c  - \ve d(t)(\mathcal L A_c)'    + (c'(t) - \ve  f_1(t)) \ve d(t)\partial_c A_c \nonu\\
& &    - (\rho'(t) -c(t) -\ve f_2(t) ) \ve d(t) A_c'  + \ve^2 d (t) f_1(t) \partial_c A_c  + O_{H^1(\R)} (\ve^2 e^{-\ve \ga |\rho(t)|} ) \nonu\\
& = & \ve d(t) (c'(t) - \ve  f_1(t)) \partial_c A_c    - \ve d(t) (\rho'(t) -c(t) -\ve f_2(t) ) A_c'  - \ve d(t)(\mathcal L A_c)' \nonu\\
& &  +\ve^2 [ \ve^{-1} d'(t) A_c + d (t) f_1(t) \partial_c A_c +\ve^{-1} d(t) \partial_t A_c ]  \label{dp}\\
& & + O_{H^1(\R)} (\ve^2 e^{-\ve \ga |\rho(t)|} ). \nonu
\eea
(Recall that $A_c'\in \mathcal S$.)  Now we use the fact that $d(t) =a_0'(\ve \rho(t))$ to compute $d'(t)$. We have
\bee
d'(t) & =&  \ve  a_0''(\ve\rho(t)) \rho'(t) \\
& =&   \ve  a_0''(\ve\rho(t))c(t) + \ve^2  a_0''(\ve\rho(t)) f_2(t)  + \ve a_0''(\ve\rho(t)) (\rho'(t) - c(t) - \ve f_2(t)).
\eee
Replacing in (\ref{dp}) we conclude.
\end{proof}

\medskip

\begin{lem}\label{lem:SintIII} Suppose that $A_c$ satisfy (\ref{Ac}). Then
\be\label{33}
{\bf III} = O_{H^1(\R)} ( \ve^2 e^{-\ve\ga|\rho(t)|}).
\ee
\end{lem}

\begin{proof}
First of all, define ${\bf \tilde{III} } := (R+ w)^p - R^p - p R^{p-1} w$. Then, 
\[
{\bf \tilde{III} }  = 
\begin{cases}
 \ve^2  d^2(t) A_c^2  & \hbox{ if } p=2;\\
 \ve^2 d^2(t) A_c^2[ 3Q_c + \ve d(t) A_c ] & \hbox{ if } p=3;\\  
\ve^2 d^2(t) A_c^2[6   Q_c^2 + 4\ve  d(t) Q_c A_c + \ve^2 d^2(t) A_c^2]&  \hbox{ in the case }  p=4.
\end{cases} 
\] 
Thus taking space derivative we obtain (\ref{33}) (note that $ (A_c^p)'\in \mathcal S$ because $A_c$ satisfies (\ref{Ac})).
\end{proof}

\medskip

\begin{center}
\line(1,0){50}
\end{center}

\medskip

Now we collect the estimates from Lemmas \ref{lem:SQ}, \ref{lem:dSKdVw} and \ref{lem:SintIII}. We obtain that, for all $t$ in a given interval,
\bea\label{Stt}
S[\tilde u] & = &   (c'(t) - \ve f_1(t))\partial_c\tilde u - (\rho'(t) -c(t) -  \ve f_2(t) )  \partial_y \tilde u + \tilde S[\tilde u], 
\eea
with
\bea\label{tStu}
& & \qquad  \tilde S[\tilde u]    =     \ve [F_1(t,y) -d(t)(\mathcal L A_{c})_y]  \\
& & \qquad \quad  + \ \ve^2 \big[ a_0'' (\ve \rho(t))c A_{c} +  f_1 a_0' (\ve \rho(t))\partial_c A_{c}  \big]  +  \ve^2a_0''(\ve \rho(t))  (\rho'(t) -c(t) -  \ve f_2(t) ) A_c \label{a09}\\
& &  \qquad \quad + \ \ve d(t) \partial_t A_c+ \ve^3  a_0''(\ve\rho(t)) f_2(t)A_c +   \ve^2O_{H^1(\R)}(e^{-\ve \ga |\rho(t)|} +\ve). \label{a10}
\eea
In addition, $f_1 (t), f_2(t)$ are given (for the moment) by (\ref{f1})-(\ref{f2}), and
\be\label{F1}
F_1  :=  F_1^{\bf I}  = f_1(t) \Lambda Q_c    - f_2(t)Q_c'  + a_0'(\ve \rho(t)) [ Q_c^2 + (c_0-c) \Lambda Q_c Q_c + (\rho_0-\rho) Q_c' Q_c],
\ee
(cf. (\ref{F1Q})). Now we give an explicit value of $f_1(t)$, satisfying (\ref{f1}). It is not difficult to check that, for any $t$, there is a well-defined $f_1(t) \in \R$ such that 
\be\label{Or}
\int_\R F_1(t, y) Q_c(y) dy= 0.
\ee
More explicitly, using (\ref{LaQc}), we have
\bee
f_1(t) &=& - a_0'(\ve \rho(t)) \Big( \int_\R \Lambda Q_c Q_c\Big)^{-1} \int_\R [ Q_c^3 + (c_0-c) \Lambda Q_c Q_c^2 ] \\
&=& f_1^0(t) -\frac{2(7-p)}{3(5-p)}  (c_0-c)a_0'(\ve \rho(t)) c^{\frac 1{p-1}} \frac{\int_\R Q^3}{\int_\R Q^2},
\eee
with $f_1^0$ defined in (\ref{f10}). This and (\ref{r2}) proves (\ref{f1}).

\medskip

The next step is the resolution of the linear differential equation involving the first order terms in $\ve$. Indeed, from (\ref{tStu}), we want to solve
\[
 d(t)(\mathcal L A_{c})_y (y) = F_1(t, y),  \quad \hbox{ for all  } y\in \R, \; \hbox{ and}  \ t \hbox{  fixed;}
\]
with $d(t)$ given by (\ref{defW}). Note that from (\ref{f1})-(\ref{f2}) and (\ref{F1}) one has
\bea\label{F1new}
F_1(t; y) &  := & a_0' [ -\la_p  c^{\frac p{p-1}} \Lambda Q_c   + Q_c^2 +(c_0-c)\Lambda Q_c Q_c + (\rho_0-\rho) Q_c' Q_c ] - f_2(t)Q_c', \nonumber \\
& = & a_0' [ -\la_p  c^{\frac p{p-1}} \Lambda Q_c  + Q_c^2  - (a_0')^{-1} f_2   Q_c' +(c_0-c)\Lambda Q_c Q_c + (\rho_0-\rho) Q_c' Q_c]  \nonu\\
& =: &d(t) \tilde F_1(t,y).
\eea 
Therefore, we are reduced to solve the following simple problem,
\be\label{Simpl}
(\mathcal L A_{c})_y (y) =  \tilde F_1(t,y),
\ee
with $\tilde F_1$ defined in (\ref{F1new}), and from (\ref{Or}),
\[
\int_\R \tilde F_1(t,y) Q_c(y) =0.
\]
Let us recall the following results (see e.g. \cite{MMcol1}):

\medskip

\begin{lem}\label{surL} The operator $\mathcal{L}$ defined (on $L^2(\R)$) by \eqref{defLy}  has domain $H^2(\R)$, it is self-adjoint and satisfies the following properties:
\begin{enumerate}
\item The kernel of $\mathcal{L}$ is spanned by $Q'_c$. Moreover, $\Lambda Q_c$ defined in (\ref{LaQc}) satisfies $\mathcal{L} (\Lambda Q_c)=- Q_c$. Finally, the continuous spectrum of $\mathcal L$ is given by $\sigma_{cont}(\mathcal L) =[c,+\infty)$.

\smallskip

\item For all   $h=h(x) $ \emph{ polynomially growing}  function such that $\int_\R h Q_c'=0$, there exists a unique  \emph{ polynomially growing}  function $\hat h $   such that $\int_\R \hat hQ'_c=0$ and $\mathcal{L} \hat h = h$. Moreover,  if $h$ is even (resp. odd), then $\hat h$ is even (resp. odd).

\smallskip

\item  For $h\in H^2(\mathbb{R})$,  $\mathcal{L} h \in \mathcal{S}(\R)$ implies $h\in \mathcal{S}(\R)$.
\end{enumerate}
\end{lem}

Let $c>0$ and
\be\label{varfi}
\varphi(x):=-\frac {Q'(x)}{Q(x)}, \qquad  \varphi_c (x) := -\frac{Q_c'}{Q_c} = \sqrt{c} \varphi(\sqrt{c} x).
\ee
Note that $\varphi$ is an odd function, with
\be\label{surphi}
\lim_{x\to \pm \infty} \varphi(x)=\pm 1; \quad \varphi^{(k)} \in \mathcal{S}(\R), \; k\geq 1.
\ee
We recall the form of the solution $A_c$ that we are looking for. We seek for a bounded solution satisfying
\be\label{eq:st}
A_{c(t)}(t,y) = \beta_c(t) (\varphi_c(y) - \sqrt{c(t)}) + \hat A_c(t,y) + \mu_c(t) Q_c'(y) + \delta_c(t) \Lambda Q_c(y),
\ee
for some $\beta_c(t), \mu_c(t),  \delta_c(t) \in \R$, $\varphi_c$ defined in (\ref{varfi}), and $ \hat A_c(t,\cdot)\in \mathcal S(\R)$. The parameters $\mu_c$ and $\delta_c$ will be chosen in order to find the \emph{unique} solution $A_c$ satisfying some orthogonality conditions. 

\medskip

\begin{lem}\label{lem:omega} Suppose $(c(t), \rho(t))$ satisfying (\ref{r1}) and (\ref{r2}), and $f_1(t)$, $f_2(t)$ given by (\ref{f1})-(\ref{f2}). There exists a unique solution $A_{c}= A_{c(t)}(t,y)$ of
\be\label{A10}
(\mathcal{L}A_{c} )_y(t,y) = \tilde F_1(t,y), 
\ee
such that, for every $t$,  
\bea\label{Acy}
& & A_{c} (t,y) :=  \beta_c(t) (\varphi_c(y) -\sqrt{c}) + \hat A_{c}(t,y) + \mu_c(t) Q_c' (y)+  \delta_c(t) \Lambda Q_c(y), \\
& & \lim_{-\infty} A_c = -2\sqrt{c} \beta_c; \quad  |A_{c}(y)|\leq K e^{-\ga y}, \; \hbox{ as } y\to +\infty,   \label{LI}
\eea
with $\hat A_{c}(t) \in \mathcal S(\R)$ for all $t$.  In addition, we have\footnote{Note that $\beta_c=0$ implies $A_c\in L^2(\R)$.}
\be\label{constants}
\beta_c(t) := \frac 1{2c^{3/2}(t)} \int_\R  \tilde F_1(t,y)dy \neq 0, \quad  |\beta_c(t)|+ |\mu_c(t)| +|\delta_c(t)| \lss 1.
\ee
Finally, $A_{c}$ satisfies
\be\label{Or3}
\int_\R A_{c} (t,y)Q_c(y)dy =\int_\R A_{c}(t,y) y Q_c(y)dy =0.
\ee
\end{lem}

\begin{proof}
First of all, the existence of a solution $A_c(t,\cdot) \in L^\infty(\R)$ of the form (\ref{Acy}) for this equation was established in \cite{Mu2}, provided 
\[
\int_\R \tilde F_1(t , y)Q_c(y)dy =0,
\]
which is indeed the case (cf. (\ref{Or})). Note  that  the inclusion of the term proportional to $f_2(t)Q_c'$ in (\ref{F1}) induces the new term $\delta_c \Lambda Q_c$ in (\ref{Acy}) (recall that from Lemma \ref{surL}  $(\mathcal L \Lambda Q_c)' =-Q_c' $.) Furthermore, the limits in (\ref{LI}) are straightforward from (\ref{surphi}).

\medskip

Now, let us prove (\ref{constants}). Indeed, from (\ref{A10}), integrating over $\R$ and using (\ref{LI}), we get
\be\label{Acminf}
2\beta_c \ c \sqrt{c}   = c A_c(-\infty)= \mathcal L A_c(+\infty) -\mathcal L A_c(-\infty)   = \int_\R\tilde F_1 ,
\ee
which gives the value of $\beta_c$, and the corresponding bound. Moreover,
\bee
\int_\R\tilde F_1 & =&  c^{2\theta} \int_\R Q^2 -\la_p c^{2\theta} (\frac 1{p-1} -\frac 12 )\int_\R Q  +(c_0-c) \theta c^{2\theta-1} \int_\R Q^2  \\
& =&  c^{2\theta} \Big[\big(\int_\R Q^2\big)^2 - \frac{2(3-p)}{(5-p)} \int_\R QÊ\int_\R Q^3 \Big] \big(\int_\R Q^2\big)^{-1} + O(|c_0-c|)\neq 0,
\eee
for $p=3,4$. The case $p=2$ requires more care, but a simple computation gives a nonzero final value: note that from the identities $Q'' =Q -Q^2$ and $Q'^2 = Q^2 -\frac 23 Q^3$, one has
\[
\int_\R Q = \int_\R Q^2, \quad \int_\R Q^3 = \frac 65\int_\R Q^2 =\frac 65\int_\R Q.
\]
Therefore $\big(\int_\R Q^2\big)^2 - \frac{2}{3} \int_\R QÊ\int_\R Q^3 = (1-\frac 45)\big(\int_\R Q\big)^2>0$.
\medskip

On the other hand, we choose the terms $\mu_c$ and $\delta_c$ in order to satisfy (\ref{Or3}). The parameter $\mu_c(t)$ is chosen to satisfy the condition
\[
\int_\R yQ_c A_c =0,
\]
and it does not give any problem. In order to deal with  $\delta_c$, we need more information about $f_2(t)$. Since we do not explicitly know $A_c$, we need another method to compute an explicit expression for $f_2(t)$, satisfying (\ref{f2}) (and therefore, the corresponding bounds for  $\delta_c(t)$).  Indeed, multiplying (\ref{A10}) by $\int_{-\infty}^y \Lambda Q_c \in L^\infty(\R)$ and integrating, one has
\be\label{Beg}
\int_\R (\mathcal L A_c)_y \int_{-\infty}^{y} \Lambda Q_c = \int_\R \tilde F_1 \int_{-\infty}^{y} \Lambda Q_c .
\ee
Integrating by parts, we get
\[
(\mathcal L A_c) \int_{-\infty}^{y} \Lambda Q_c \Big|_{-\infty}^{+\infty} +  \int_\R (\mathcal L A_c)_y \int_{-\infty}^{y} \Lambda Q_c = -\int_\R \Lambda Q_c \mathcal L A_c  =\int_\R Q_c A_c =0.
\]
Using (\ref{Ac}), we have $(\mathcal L A_c) \int_{-\infty}^{y} \Lambda Q_c \Big|_{-\infty}^{+\infty} =0$. Therefore, from (\ref{F1new}), 
\bee
& & -f_2 \int_\R Q_c \Lambda Q_c = \\
& & \qquad = a_0'\int_\R \Big[ -\la_p  c^{\frac p{p-1}}  \Lambda Q_c  + Q_c^2 +(c_0-c)\Lambda Q_c Q_c + (\rho_0-\rho) Q_c' Q_c  \Big]\int_{-\infty}^y \Lambda Q_c.
\eee
A simple computation using the scaling of $Q_c$, $\Lambda Q_c$ and its derivatives, and integration by parts show that, for  $\theta = \frac1{p-1}-\frac 14$,
\bee
- \theta  f_2 c^{2\theta -1} \int_\R Q^2  & = &  a_0' \Big[  -\frac 12\la_p  c^{\frac p{p-1}}   \big(\int_\R \Lambda Q_c\big)^2  +\int_\R  Q_c^2 \int_{-\infty}^y \Lambda Q_c  \\
& & \qquad +(c_0-c) \int_\R \Lambda Q_c Q_c\int_{-\infty}^y \Lambda Q_c -\frac12 (\rho_0-\rho)\int_\R Q_c^2 \Lambda Q_c \Big] \\
&  =&  \frac{(3-p)}{2(p-1)} a_0' c^{\frac{5-2p}{p-1}} \Big[ -\frac 14 \la_p\frac{(3-p)}{p-1}   \big(\int_\R Q\big)^2  +\int_\R Q^2 \int_{-\infty}^y Q  \Big] \\
& & +(c_0-c)a_0'  c^{\frac{3(2-p)}{p-1}} \int_\R Q\Lambda Q \int_{-\infty}^y \!\!\! \Lambda Q  \\
& &  -\frac{(7-p)}{12(p-1)} (\rho_0-\rho) a_0' c^{\frac{3(3-p)}{2(p-1)}} \int_\R Q^3 .
\eee
We finally obtain
\bee
& & f_2(t) =\\
&  & \quad = \frac{2(p-3)}{(5-p)} a_0' (\ve \rho(t)) [c(t)]^{\frac{2(5-2p)}{7-3p}} \Big[ -\frac 14 \la_p\frac{(3-p)}{p-1}   \big(\int_\R Q\big)^2  +\int_\R Q^2 \int_{-\infty}^y Q\Big] \big(\int_\R Q^2\big)^{-1}\\
& & \quad  -\frac{5-p}{4(p-1)}(c_0(t)-c(t)) a_0' (\ve \rho(t))  [c(t)]^{\frac{6(2-p)}{7-3p}}  \big(\int_\R Q\Lambda Q \int_{-\infty}^y \!\!\! \Lambda Q \big) \big(\int_\R Q^2\big)^{-1}  \\
& & \quad  + \frac{(7-p)}{3(5-p)} (\rho_0(t)-\rho(t)) a_0' (\ve \rho(t)) [c(t)]^{\frac{3(3-p)}{7-3p}}  \big(\int_\R Q^3 \big)  \big(\int_\R Q^2\big)^{-1} \\
& & \quad = :\mu_p a_0'(\ve \rho(t)) [c(t)]^{\frac{2(5-2p)}{7-3p}}  +  O(a_0' (\ve \rho(t))  (|c_0(t)-c(t)| +  |\rho_0(t)-\rho(t)| ) ),
\eee
as desired\footnote{Note that the exponents $\frac{6(2-p)}{7-3p}$ and $\frac{3(3-p)}{7-3p}$ are both nonnegative for $p=2,3$ and 4.} (cf. (\ref{f2}) and (\ref{f20}), and note that $\mu_3=0$). Note that from (\ref{f2}) we have that $\delta_c(t)$ satisfies the required estimates.

\end{proof}

\medskip

Having solved the  linear problem, from (\ref{Stt}) and (\ref{tStu}) we have
\[
S[\tilde u](t,x)  =    (c'(t) - \ve f_1(t) )\partial_c\tilde u - (\rho'(t) -c(t) -  \ve f_2(t))  \partial_y \tilde u + \tilde S[\tilde u](t,x), 
\]
where $\tilde S[\tilde u]$ given in (\ref{a09})-(\ref{a10}) will be of second order in $\ve$, as we show in the following lines.

\medskip

Let us describe the dependence on $c$ and $t$ of the solution $A_c$. From (\ref{F1new}) (see also Lemma 4.5 in \cite{Mu2}), one has
\[
\tilde F_1 (t,y)  = c^{\frac 2{p-1}}  \tilde F_1^1(\sqrt{c}y)  + c^{\frac{32p-13-11p^2}{2(p-1)(7-3p)}}  \tilde F_1^2(\sqrt{c}y) + O_{H^1(\R)}(|c-c_0| + |\rho-\rho_0|), 
\]
where
\[
\tilde F^1_1(y)  :=  -\la_p  \Lambda Q(y)   +Q^2(y), \quad \tilde F^2_1(y) := Q'.
\]
Therefore, Claim 3 in \cite{Mu2} allows to conclude that $A_c$ satisfies the following decomposition:
\be\label{ScaA}
A_c(t,y) = c^{\frac {7-3p}{2(p-1)}}  \tilde A_1^1(t,\sqrt{c}y) + c^{\frac{4 +p-p^2}{(p-1)(7-3p)}}  \tilde A_1^2(t,\sqrt{c}y) +  O_{L^\infty(\R)}(|c-c_0| + |\rho-\rho_0|), 
\ee
with $ \tilde A_1^1$ bounded solution of $(\mathcal L \tilde A_1^1)' =  \tilde F_1^1$, and $ \tilde A_1^2(\sqrt{c}y)\in \mathcal S(\R)$. Moreover, one has $ (\tilde A_1^1)' \in \mathcal S(\R)$. Using this decomposition we have that, avoiding the terms proportional to  $|c-c_0|$, $\partial_c A_c$ has the same behavior as $A_c$: it is bounded, it is not $L^2$-integrable, and satisfies $\lim_{+\infty} \partial_c A_c =0$, $\lim_{-\infty} \partial_c A_c \neq0$. The same result holds for $\partial_c^2 A_c$.

\medskip

We consider now the term $\partial_t A_c$, avoiding the terms with usual derivatives with respect to $c$ and $\rho$.
In fact, $\partial_t A_{c}(t,y)$ involves derivatives with respect to $t$ of $(c_0-c)$ and $(\rho_0-\rho)$. More specifically, from the explicit composition of $F_1$ in (\ref{F1new}) the solution $A_c$ can be decomposed as follows
\bee
A_c(t,y) & =&  A_{c,s}(t,y)  + a_0'(\ve \rho(t)) (\rho_0(t)-\rho(t)) D_c(t,y) \\
& & + a_0'(\ve \rho(t)) (c_0(t)-c(t)) E_c(t,y),  
\eee
where $A_{c,s}$ is the solution of 
\[
(\mathcal L A_{c,s})_y = a_0' [ -\la_p  c^{\frac p{p-1}} \Lambda Q_c+ Q_c^2 ],
\]
$D_c$ solves 
\[
\mathcal LD_c = \frac 12 Q_c^2  - (a_0')^{-1} f_2   Q_c,
\]
and $E_c$ is the solution of 
\[
(\mathcal L E_c)_y = \Lambda Q_c Q_c.
\]
It is clear that  $D_c (t,\cdot)\in \mathcal S(\R)$ and $E_c(t,\cdot) \in L^\infty(\R)$. Since the term $A_{c,s}(t,y)$ only contains derivatives in time already computed in (\ref{Stt}),  we get
\be\label{Marca}
\partial_t A_c = a_0'(\ve \rho(t)) (\rho_0'(t)-\rho'(t)) D_c(t,y) + a_0'(\ve \rho(t)) (c_0'(t)-c'(t)) E_c(t,y).
\ee
Note that 
\bee
c_0'-c' & =& \ve (f_1^0(c_0,\rho_0) -f_1(c,\rho)) + O(|c' -\ve f_1|) \\
& =&  O(\ve e^{- \ga\ve|\rho_0|}+\ve e^{- \ga\ve|\rho|} + |c' -\ve f_1|).
\eee
Note that the second term can be added to the dynamical system (\ref{Stt}) without perturbing the dynamics. The worst case is with no doubt the first one. We have
\[
|\rho_0'-\rho'| \lss |c_0-c| + |\rho' -c -\ve f_2| + O(\ve).    
\]
The term $|\rho' -c -\ve f_2| $ can be added to the dynamical system (\ref{Stt}) as in the previous case. In concluding, without considering the terms proportional to $|\rho' -c -\ve f_2|$ and $|c'-\ve f_1|$,
\bee
\tilde S[\tilde u] & =&  (\ref{a09}) + (\ref{a10}) \\
& =&  O(\ve^2 e^{- \ga\ve|\rho_0|}+\ve^2 e^{- \ga\ve|\rho|} ) + \ve a_0'(\ve \rho(t)) (c_0(t)-c(t)) D_c(t,y),
\eee
with exponential decay as $y\to +\infty$, and $D(t,\cdot) \in \mathcal S(\R)$. These estimates will be useful when computing (\ref{a03}).

\medskip

Let us conclude the proof of Proposition \ref{prop:decomp}. Using the decay on the right of $A_c$ (see (\ref{LI})), estimate (\ref{Rig}) is direct. In addition, from Lemma \ref{lem:omega} we have (\ref{Ac}), and  $f_1(t)$ and $f_2(t)$ are well determined by (\ref{f1})-(\ref{f2}). Finally, from (\ref{a09})-(\ref{a10}) one has (\ref{SIn0}). These facts prove Proposition \ref{prop:decomp}.

%
%

\end{proof}

The next results are similar to those proved in \cite{Mu2, Mu3}, but for the sake of completeness, we include them. Recall that $\tilde u$ does not belong to $L^2(\R)$. In order to solve this problem, consider a cut-off function $\eta \in C^\infty (\R)$ satisfying the following properties:
\be\label{eta}
\begin{cases}
0\leq \eta (s) \leq 1, \quad 0\leq  \eta' (s) \leq 1, \; \hbox{ for any } s\in \R;\\
\eta(s)\equiv 0 \; \hbox{ for } s\leq -1, \quad  \eta(s)\equiv 1 \; \hbox{ for } s\geq 1,
\end{cases}
\ee 
and define 
\be\label{etac}
\eta_\ve (y) := \eta( \ve y +  2 ),
\ee
and for $w=w( t, y)$ the first order correction constructed in Lemma \ref{lem:omega}, \emph{redefine}
\be\label{hatu}
\tilde u(t,x) :=  \eta_\ve (y) \tilde u(t,x) =\eta_\ve (y) (R(t,x) + w(t,x)),
\ee
and similarly for $R(t)$ and $w(t)$. Note that, by definition, 
\be\label{newb}
\tilde u(t, x) = 0 \quad \hbox{ for all } y\leq -\frac 3\ve.
\ee
The following Proposition deals with the error associated to this cut-off function, and the new approximate solution $\tilde u$.

\medskip

\begin{prop}\label{CV} There exist constants $\ve_0,\ga>0$ such that for all $0<\ve <\ve_0$ the following holds. 

\medskip

\begin{enumerate}

\item Consider the localized function $\tilde u(t) = R(t) +w(t)$ defined in (\ref{etac})-(\ref{hatu}), for $t$ in a given interval. Then we have

\begin{enumerate}
\item \emph{$L^2$-solution}. $w(t, \cdot ) \in H^1(\R)$, with
\be\label{H1}
\|w(t, \cdot ) \|_{H^1(\R)} \lss \sqrt{\ve} e^{-\ga \ve |\rho(t)|}.
\ee
\item \emph{Almost orthogonality}. 
\be\label{AO}
\abs{\int_\R w(t,x)Q_c(y)dx} + \abs{\int_\R yw(t,x)Q_c(y)dx} \lss \ve^{10}.
\ee
\end{enumerate}

\item \emph{Almost solution}. The error associated to the new function $\tilde u(t)$ satisfies
\bee
S[\tilde u]  &= & (c'(t) - \ve f_1(t) ) (\partial_c\tilde u + O_{H^1}(\ve a_0'(\ve\rho) |c_0-c|))   \\
&&  - (\rho'(t) -c(t)  -  \ve f_2(t) )  (\tilde u_y + O_{H^1}(\ve a_0'(\ve\rho) (|c_0-c| + |\rho_0-\rho|))) \\
& &+ \tilde S[\tilde u](t),
\eee
with 
\be\label{Stilde}
\| \tilde S[\tilde u](t) -\ve a_0'(\ve \rho(t)) (c_0(t)-c(t)) D_c(t) \|_{H^1(\R)} \lss \ve^{3/2}e^{-\ga \ve |\rho(t)|}.
\ee
Finally, one has (\ref{SIn0}).
\end{enumerate}

\end{prop}

\medskip

\begin{proof}
The proof of (\ref{H1}) follows from a direct computation. Indeed,
\[
\|w(t) \eta_\ve\|_{H^1(\R)} \lss \|w(t)\|_{H^1(y\geq -\frac 3\ve)} ,
\]
but from (\ref{Ac}),
\[
\| \ve d(t) A_c(y) \|_{H^1(y\geq -\frac 3\ve)}  \lss \sqrt{\ve}e^{-\ve\ga|\rho(t)|}.
\]
Let us now consider (\ref{AO}). Here we have, using (\ref{Ac}), 
\[
\int_\R w(t,x) \eta_\ve (y) Q_c(y) = \int_\R w(t,x) (\eta (\ve y +2)-1) Q_c(y). 
\]
Note that $\eta (\ve y +2)-1 \equiv 0$ for $y\geq -\frac 1\ve$. Using the exponential decay of $Q_c(y)$, we have
\bee
\abs{\int_\R w(t,x) \eta_\ve (y) Q_c(y)}  & \lss &  \int_{y\leq -\frac 2\ve}  \ve |y| e^{\sqrt{c} y} + \int_{y\in (-\frac 2\ve, -\frac 1\ve)}  \ve |y| e^{- \frac 12 (\ve y+2)} e^{\sqrt{c} y} \\
& \lss &  e^{-\ga/\ve} \ \lss \ve^{10}. 
\eee
The proof for $yA_c$ is very similar. We skip the details. For the proof of (\ref{Stilde}), we proceed as follows. First of all, a simple computation shows that
\[
S[\eta_\ve \tilde u]  =   \eta_\ve S[\tilde u] + (\eta_\ve)_t \tilde u + 3\ve \eta_\ve' \tilde u_{xx} + 3\ve^2 \eta_\ve^{(2)} \tilde u_{x} + \ve^{3} \eta_\ve^{(3)} \tilde u  + \ve (\eta_\ve^p)'  \tilde u^p.
\]
Since $\supp \eta_\ve^{(k)} \subseteq [-\frac 3\ve, -\frac 1\ve]$ for $k=1,2$ and $3$, we have
\[
3\ve \eta_\ve' \tilde u_{xx} + 3\ve^2 \eta_\ve^{(2)} \tilde u_{x} + \ve^{3} \eta_\ve^{(3)} \tilde u  + \ve (\eta_\ve^p)'  \tilde u^p =  O_{H^1(\R)} (\ve^{3/2} e^{-\ve\ga|\rho(t)|}) + O_{H^1(\R)}(\ve^{10}).
\]
Similarly, from the definition of $\rho'(t)$ and (\ref{r1})
\[
(\eta_\ve)_t \tilde u  =  -\rho'(t)\ve \eta_\ve' \tilde u  =   O_{H^1(\R)}(\ve^{3/2} e^{-\ve\ga|\rho(t)|}) + O_{H^1(\R)}(\ve^{10}).
\]
Collecting the terms above, we have
\[
S[\eta_\ve \tilde u]  =   \eta_\ve S[\tilde u]+ O_{H^1(\R)}(\ve^{3/2} e^{-\ve\ga|\rho(t)|}) + O_{H^1(\R)}(\ve^{10}).
\]
Finally, from the decomposition (\ref{Sdeco}), one has $S[\tilde u] = \hbox{\emph{dynamical system}} + \tilde S[\tilde u]$, and from (\ref{Stilde}), (\ref{defW}) and (\ref{Ac}),
\[
\| \eta_\ve \tilde S[\tilde u] \|_{H^1(\R)} \lss \ve^{3/2} e^{-\ve\ga|\rho(t)|} + \ve |c(t)-c_0(t)|  e^{-\ve\ga|\rho(t)|}+  \ve^3.
\]
Note that the proof of (\ref{SIn0}) does not vary at all. Finally, one has
\bee
 \eta_\ve \Big[  (c' - \ve f_1)\partial_c\tilde u    +\ (\rho' -c -  \ve f_2 ) \partial_\rho \tilde u \Big]   &  =&  (c' - \ve f_1)\partial_c(\eta_\ve\tilde u) -  (\rho' -c -  \ve f_2 ) \partial_y(\eta_\ve \tilde u) \\
& &   + \ \ve (\rho' -c-  \ve f_2 ) \eta_\ve ' \tilde u.
\eee
Since $\ve \eta_\ve ' \tilde u =O_{H^1(\R)}(\ve^{3/2} e^{-\ve\ga|\rho(t)|}) $, from this last estimate, we get the final conclusion. 
\end{proof}
%
Finally, we recall that
\be\label{SSS}
\| \tilde S[\tilde u](t)  - \ve a_0'(\ve \rho(t)) (c_0(t)-c(t)) D_c(t,y) \|_{H^1(\R)} \lss \ve^{3/2} e^{-\ve\ga|\rho(t)|} .
\ee
with $|D_c(t,y) | + |\partial_xD_c(t,y) |\lss e^{- \ga_0|y|}$, for some fixed constant $\ga_0>0$.

\bigskip

\section{Lyapunov stability}\label{4} 
 
\medskip
 
In this section we prove the following 
 
\medskip

\begin{prop}\label{prop:I} The following holds for any $0<\ve <\ve_0$. There exist $K_0>0$ independent of $\ve$ and unique $C^1$ functions $c, \rho : [0, T] \to \R$ such that, for all $t\in [0, T]$,
\be\label{INT41}
\|u(t)-\tilde u(t,c(t),\rho(t)) \|_{H^1(\R)} \leq K_0 \sqrt{\ve},
\ee
\[
|c(t) -c_0(t)| +|\rho(t) -\rho_0(t)| \lss K_0 \sqrt{\ve}, 
\]
and
\bea\label{INT42}
& & | \rho'(t) -c(t)  -\ve f_2(t)| +\ \frac{1}{\sqrt{\ve}} |c'(t) - \ve f_1(t) | \lss K_0 \sqrt{\ve}.
\eea
Finally, one has
\be\label{crm}
|c(0) -1 | + |\rho(0) + \ve^{1+\delta_0} | \lss \sqrt{\ve},
\ee
with constants independent of $K_0$.
\end{prop}

A direct conclusion of the previous result is the following

\medskip

\begin{cor} 
Let $T_0>0$ be the maximal time of existence of $u(t)$. Then
\[
T_0 \geq  \ve^{-1-\delta_0}.
\]
\end{cor}

\begin{rem}[Notation]
For the sake of brevity, in the forthcoming computations, we will denote 
\[
c_1' := c'-\ve f_1, \quad \hbox{ and } \quad \rho_1' := \rho' -c - \ve f_2.
\]
\end{rem}

\begin{proof}[Proof of Proposition \ref{prop:I}] Let $K^*>1$ be a constant to be fixed later. Since $u(0)= Q(x)$, by the local continuity in $H^1(\R)$ of the gKdV flow, there exists a time $T^*>0$  such that, for all $t\in [0, T^*]$,  we can find continuous functions $\la(t), r(t)\in \R$, such that 
\be\label{Tstar}
 \|u(t) -  \tilde u( \cdot \ , \la(t), r(t))\|_{H^1(\R)}\leq K^* \sqrt{\ve},
\ee
\be\label{Tstar1}
|\la(t)- c_0(t)| +| r(t) -\rho_0(t)| \lss  K^* \ve^{1/2(1-\delta_0)},
\ee
and 
\[
|\la(0) - 1| + |r(0) + \ve^{1+\delta_0}| \lss \sqrt{\ve},
\]
with a constant independent of $K^*$ large. Without loss of generality, we can assume $T^* = T^*(K^*)$ as the least upper bound of times such that the properties above are satisfied. The objective is to prove that we can take $T^*\geq  T$ for $K^*$ large enough, by proving a bootstrap estimate for suitable well chosen parameters $\la(t)$, $r(t)$. Our first step is to choose such parameters.

\medskip

\begin{lem}\label{DEFZ} Assume $0<\ve<\ve_0(K^*)$ small enough. There exist unique $C^1$ functions $c(t), \rho(t)$ such that, for all $t\in [0,T^*]$,
\be\label{defz}
z(t)=u(t)-\tilde u(t,c(t),\rho(t)) \; \hbox{satisfies}\; \int_\R z(t,x) y Q_c(y) dx  = \int_\R z(t,x) Q_c(y)dx=0.
\ee
Moreover, we have,  for all $t\in [0,T^*]$,
\be\label{TRANS3}
 \|z(0)\|_{H^1(\R)} + | c(0) -1 | + |\rho(0)+ \ve^{1+\delta_0}| \lss \sqrt{\ve}, 
\ee
\be\label{TRANS4}
  \|z(t)\|_{H^1(\R)}  \lss K^* \sqrt{\ve},  \qquad  |c(t) -c_0(t)| +|\rho(t) - \rho_0(t)| \lss K^* \sqrt{\ve}.
\ee
In addition, $z(t)$ satisfies the following equation
\be\label{Eqz1}
 z_t +  \big\{ z_{xx}  +  (\tilde u +z)^m - \tilde u^m  \big\}_x + \ve a_0'(\ve x)Q_{c_0} z  +  \tilde S[\tilde u]     + c_1'(t) \partial_c \tilde u  - \rho_1'(t) \partial_y \tilde u =0.
\ee
Finally, there exists $\ga>0$ independent of $K^*$ such that for every $t\in [0, T^*]$,
\be\label{rho1}
 |\rho_1'(t) |  \lss   \Big[\int_\R z^2 e^{- \ga\sqrt{c}|y|} \Big]^{1/2}  +  \int_\R e^{- \ga\sqrt{c}|y|}z^2(t) + \abs{\int_\R yQ_c \tilde S[\tilde u]},
\ee
and
\be\label{c1}
|c_1'(t) |\lss  \int_\R e^{-\ga\sqrt{c} |y|} z^2(t)  +  \ve e^{-\ga\ve|\rho(t)| } \Big[ \int_\R e^{-\ga\sqrt{c} |y|} z^2(t)\Big]^{1/2} + \abs{\int_\R Q_c \tilde S[\tilde u]}.
\ee
\end{lem}

\begin{proof}
The proof of (\ref{defz})-(\ref{TRANS3}) is a standard consequence of the Implicit Function Theorem, applied for each time $t\in [0, T^*]$. Indeed, for fixed $t\in [0,T]$, let us define the map 
\[
(v,c,\rho) \in H^1(\R) \times \R_+\times \R  \longmapsto J(v,c,\rho) \in \R^2,
\]
where $J =(J_1,J_2)$ and
\bee
J_1(v,c,\rho) & :=&   \int_\R (v-\tilde u(t,c,\rho))(x-\rho)Q_c(x-\rho)dx,\\
 J_2(v,c,\rho) & :=&   \int_\R (v-\tilde u(t,c,\rho))Q_c(x-\rho)dx.
\eee
It is clear that 
\[
J_1( \tilde u(t,c_0(t),\rho_0(t)),c_0(t),\rho_0(t)) = J_2( \tilde u(t,c_0(t),\rho_0(t)),c_0(t),\rho_0(t)) \equiv 0.
\]
Moreover, the respective Jacobian determinant of $J$ with respect to the variables $(c,\rho)$ is nonzero everywhere. Therefore, from the Implicit Function Theorem, there exists a small $\eta_0(t)>0$ (which can be chosen continuous on $t$) such that, for all $v \in H^1(\R)$ satisfying $\|v - \tilde u(t,c_0(t),\rho_0(t)) \|_{H^1(\R)} < \eta_0$, there is a smooth pair of parameters $(c(v),\rho(v)) \in \R^2$, satisfying $J(v,c(v),\rho(v)) \equiv 0.$ Since the interval $[0,T^*]$ is compact, we can ensure $\eta_0>0$ independent of $t$. 

\medskip

Note that from (\ref{Tstar})-(\ref{Tstar1}), the function $u(t)$ satisfies, for each $t\in [0,T]$,
\be\label{Cond0}
\|u(t) - \tilde u(t,c_0(t),\rho_0(t)) \|_{H^1(\R)} \lss K_0\sqrt{\ve} \ll \eta_0,
\ee
provided $\ve_0=\ve_0(\eta_0)$ is chosen even smaller. Therefore there exists a smooth pair of parameters $(c(u(t)),\rho(u(t))) =:(c(t),\rho(t)) \in \R^2$, such that $J(u(t),c(t),\rho(t)) \equiv 0.$ This proves (\ref{defz}).

\medskip

The proof of (\ref{TRANS3}) is direct from the initial condition $u_0(x)=Q(x)$, and (\ref{r1}). Finally, (\ref{TRANS4}) follows from  (\ref{Cond0}) and the fact that $J(u(t),c(t),\rho(t)) \equiv 0.$

\medskip

On the other hand, (\ref{Eqz1}) is a direct computation. For the proof of (\ref{rho1}) and (\ref{c1}), see e.g. \cite{Mu3}. 
\end{proof}
From (\ref{SIn0}) and (\ref{TRANS3}), a crude estimate of the parameter $|c_1'(t)|$ gives $|c_1'(t)| \lesssim \ve$, so that
\[
\int_0^{T^*} |c_1'(t)|dt \lesssim \ve^{-\delta_0},
\]
which is not a good estimate. In the following lines, we prove a sharp Virial estimate \cite{MMnon,Mu3} which allows to improve the preceding bound.

\medskip

\begin{center}
\line(1,0){50}
\end{center}

\medskip
 
First of all, we define some auxiliary functions. Let $\phi \in C^\infty(\R)$ be an \emph{even} function satisfying the following properties
\be\label{psip}
\begin{cases}
\phi' \leq 0 \; \hbox{ on } [0, +\infty); \quad  \phi (x) =1 \; \hbox{ on } [0,1], \\
\phi (x) = e^{-x}  \; \hbox{ on } [2, +\infty) \quad\hbox{and}\quad  e^{-x} \leq \phi (x) \leq 3e^{-x}  \; \hbox{ on } [0,+\infty).
\end{cases}
\ee
Now, set $\psi(x) := \int_0^x \phi $. It is clear that $\psi$ an odd function. Moreover, for $|x|\geq 2$,
\be\label{asy}
\psi(+\infty) -\psi (|x|) = e^{-|x|}.
\ee
Finally, for $A>0$, denote 
\be\label{psiA}
\psi_A(x) := A\psi(\frac xA); \quad e^{-|x|/A} \leq \psi_A'(x)   \leq 3e^{-|x|/A}. 
\ee
We claim the following 

\medskip

\begin{lem}\label{VL} There exist $K, A_0, \delta_0>0$  such that for all $t\in [0, T^*]$,
\be\label{dereta}
 \partial_t \int_\R  z^2(t,x) \psi_{A_0}(y)   \leq   -\delta_0  \int_\R ( z_x^2 + z^2 )(t,x) e^{-\frac 1{A_0} |y|}  + KA_0 (K^*)^{p+1} \ve^{5/2}.
\ee
\end{lem}

\begin{proof} Let $t\in [0, T^*]$. Replacing the value of $ z_t$ given by (\ref{Eqz1}),  we have
\bea
 \partial_t \int_\R  z^2 \psi_{A_0}(y)  & = &   2\int_\R  z  z_t  \psi_{A_0}(y) -\rho'(t) \int_\R  z^2 \psi_{A_0}'(y) \nonumber \\
&  = &  - 2\ve \int_\R a_0'(\ve x) Q_{c_0} z^2 \psi_{A_0}(y)  + 2\int_\R ( z   \psi_{A_0}(y) )_x (  z_{xx}+  p  R^{p-1}  z )  \label{e0} \\
& &    -(c +\ve f_2)(t)\int_\R  z^2 \psi_{A_0}' - 2\rho'_1(t)\int_\R  z \partial_\rho \tilde u \psi_{A_0}\label{e1} \\
& &   + 2\int_\R ( z   \psi_{A_0}(y) )_x   [(\tilde u +  z)^p -\tilde u^p -p\tilde u^{p-1} z ] \label{e2}\\
& &  - 2c_1'(t)\int_\R  z \partial_c\tilde u \psi_{A_0}    - \rho_1'(t) \int_\R  z^2 \psi_{A_0}' \label{e3}\\
& &   + 2p\int_\R z( z   \psi_{A_0}(y) )_x  (\tilde u^{p-1} - R^{p-1}) -2\int_\R z \psi_{A_0}\tilde S[\tilde u] \label{e4}.
\eea
The first term in (\ref{e0}) can be estimated as follows:
\[
\abs{\ve \int_\R a_0'(\ve x)Q_{c_0} z^2 \psi_{A_0}(y)} \lss\ve A_0 \int_\R z^2(t)e^{-\frac 1{A_0}|y|} .
\]
On the other hand, note that
\bee
|(\ref{e2})| & \lss & \abs{ \int_\R  z_x   \psi_{A_0}(y)   [(\tilde u +  z)^p -\tilde u^p -p \tilde u^{p-1} z ]}  \\
& & \quad  +  \abs{\int_\R     \psi_{A_0}'(y) z [(\tilde u +  z)^p -\tilde u^p -p\tilde u^{p-1} z ] } \\
& \lss &   A_0 K^*\ve^{1/2}\int_\R   z^2 (t)e^{-\ga\sqrt{c}|y|} + K^*\ve^{1/2} \int_\R   z^2(t)e^{-\frac 1{A_0}|y|}  + \abs{\int_\R z^{p+1}  \psi_{A_0}'(y) }\\
& \lss & K^*A_0 \ve^{1/2}\int_\R   z^2(t)e^{-\frac 1{A_0}|y|} + A_0 \ve \|z(t)\|_{H^1(\R)}^{p+1} \\
& \lss & K^*A_0 \ve^{1/2}\int_\R   z^2(t)e^{-\frac 1{A_0}|y|} +  (K^*)^{p+1}A_0 \ve^{(p+3)/2}.
\eee
for $A_0$ large, but independent of $\ve$.
Now, by using (\ref{rho1}) and (\ref{c1})  it is easy to check that for $A_0$ large enough, and some constants $\delta_0, \ve_0$ small, one has
\bee
 |(\ref{e3})|  &  \lss &   |c_1'(t)| \abs{\int_\R  z \partial_c\tilde u \psi_{A_0} }+  K^* \ve^{1/2} \int_\R  z^2(t) e^{-\frac 1{A_0}|y|} \\
 & \leq &  \frac{\delta_0}{100}\int_\R z^2(t)e^{-\frac 1{A_0}|y|} + KK^* A_0\ve^{5/2} e^{-\ve\ga|\rho(t)|}.
\eee
On the other hand, the terms (\ref{e0}) and (\ref{e1}) goes similarly to the terms $B_1$ and $B_2$ in Appendix B of \cite{MMnon}. Indeed, we have
\bee
(\ref{e0}) + (\ref{e1}) & =& -\int_\R \psi_{A_0}' (  3z_x^2  + c z^2 -  pQ_c^{p-1}  z^2 ) - p\int_\R (Q_c^{p-1})' z^2 \psi_{A_0}  \\
& &   + \int_\R  z^2 \psi_{A_0}^{(3)} - 2\rho_1'(t) \int_\R  z \partial_\rho \tilde u\psi_{A_0}  \\
& &  +2p\int_\R (z\psi_{A_0})_x z (R^{p-1} -Q_c^{p-1}) -\ve f_2 \int_\R z^2 \psi_{A_0}' .
\eee
We finally get, taking $\ve$ small, depending on $A_0$,
\[
(\ref{e0}) + (\ref{e1})\leq  -\frac{\delta_0}{10}\int_\R ( z_x^2 +  z^2)(t)e^{-\frac 1{A_0}|y|}.
\]
Finally, the term (\ref{e4}) can be estimated as follows
\bee
|(\ref{e4})| & \lss &  \abs{\int_\R z( z   \psi_{A_0}(y) )_x (\tilde u^{p-1} -  R^{p-1}) } + \abs{\int_\R z \psi_{A_0}\tilde S[\tilde u]}\\
& \lss & \abs{\int_\R z^2   \psi'_{A_0}(y) (\tilde u^{p-1} - R^{p-1}) }   \\
& & \quad +  \abs{\int_\R z z_x   \psi_{A_0}(y)  (\tilde u^{p-1} - R^{p-1}) } +   A_0 (K^*)^2 \ve^{5/2} e^{-\ve \ga|\rho(t)|} \\
& \lss &  A_0 \ve \int_\R (z^2(t) + z_x^2(t)) e^{-\frac 1{A_0} |y|}  +  A_0 (K^*)^2 \ve^{5/2} e^{-\ve \ga|\rho(t)|}.
\eee
We have used that $\psi_{A_0}$ decreases exponentially as $y\to -\infty$, and (\ref{SSS}). Collecting these estimates, we finally get (\ref{dereta}).
\end{proof}

\medskip

\begin{cor} One has, from (\ref{c1}) and (\ref{dereta}),
\be\label{intc1}
\int_{0}^{t} |c'_1(s)| ds  \lesssim K^* \ve,
\ee
for all $t\in [0, T^*]$, by taking $A_0$ large enough, independent of $\ve$ and $K^*$.
\end{cor}

\medskip

The main part of the proof is the introduction of the following Lyapunov functional (\cite{MMcol1,Mu2}): Let
\be\label{F}
\mathcal F(t) := \frac 12 \int_\R (z_x^2 +c(t) z^2) - \frac{1}{p+1}\int_\R  [(\tilde u+ z)^{p+1} -\tilde u^{p+1} - (p+1)\tilde u^{p}z]. 
\ee
From \cite{MMcol1} and the fact that $p<5$, there exists a constant, independent of $K^*$ and $\ve$ such that for every $t\in [0, T^*]$
\be\label{Coer2}
\mathcal F(t) \gtrsim \|z(t)\|_{H^1(\R)}^2 . 
\ee
The next step is to obtain independent estimates on $\mathcal F(t)$.  We follow \cite{Mu3}. It is not difficult to check that 
\bea\label{Fp}
\mathcal F'(t) & =&     -\int_\R z_t [ z_{xx} -c z +   (\tilde u+z)^p -\tilde u^p ] + \frac 12 c'(t)\int_\R z^2  \nonu\\
& & -\int_\R  \tilde u_t[ (\tilde u+z)^p -\tilde u^p -p\tilde u^{p-1}z]. 
\eea

\begin{lem}\label{Ka} There exists a constant $\ga>0$ such that, for any $t\in [0, T^*]$,
\bea\label{IntF}
\mathcal F(t) -\mathcal F(0) & \lss &  (K^*)^4 \ve^{2-\frac 1{100}}  + (K^*)^3 \ve^{\frac 32-\frac 1{100}} \nonu\\
&& + K^* \ve   + \int_{0}^t \ve e^{-\ve\ga |\rho(s)|} \|z(s)\|_{H^1(\R)}^2 ds . 
\eea
\end{lem}

\begin{proof}
Replacing (\ref{Eqz1}) in (\ref{Fp}) we get
\bea
& & \mathcal F'(t) =\nonu\\
&  &=  \ve \int_\R a_0'(\ve x) Q_{c_0} z [ z_{xx} -c z +   (\tilde u+z)^p -\tilde u^p ]   + c(t)  \int_\R  [ (\tilde u+z)^p -\tilde u^p ]  z_x \label{Fp1} \\
& &\quad - \rho_1'(t)  \int_\R \partial_y \tilde u [z_{xx} -cz + (\tilde u+z)^p -\tilde u^p ] \label{Fp2} \\
& &\quad + c_1'(t) \int_\R \partial_c \tilde u [ z_{xx} -cz +  (\tilde u+z)^p -\tilde u^p  ]\label{Fp2b} \\
& &\quad +  \int_\R \tilde S[\tilde u] [ z_{xx} -cz +  (\tilde u+z)^p -\tilde u^p ]  + \frac 12 c_1'(t) \int_\R z^2+\frac 12 \ve f_1(t) \int_\R z^2 \label{Fp3b} \\
& &\quad -\int_\R  \tilde u_t [ (\tilde u+z)^p -\tilde u^p -p\tilde u^{p-1}z]. \label{Fp4}
\eea 
We consider the case $p=2$, the other cases being similar (see \cite{Mu2} for more details). First of all, note that 
\[
\abs{ \ve \int_\R a_0'(\ve x) Q_{c_0} z [ z_{xx} -c z +   (\tilde u+z)^p -\tilde u^p ]  } \lss \ve e^{-\ga \ve |\rho_0(t)|} \|z(t)\|_{H^1(\R)}^2.
\]
Next, after some simplifications, we get
\[
 c(t)  \int_\R  [ (\tilde u+z)^p -\tilde u^p ]  z_x    =  c(t)  \int_\R  [2\tilde u z + z^2 ] z_x  =  -c(t) \int_\R   \tilde u_x z^2   .
\]
Now we estimate (\ref{Fp2}). Since $\partial_y \tilde u =Q_c' + O(\eta_\ve w_y) + O_{H^1(\R)}(\ve^{3/2} e^{-\ve \ga |\rho(t)|})$ (cf. Proposition \ref{CV}), one has 
\bea\label{a01}
(\ref{Fp2})  & = &  -\rho'_1(t) \int_\R \partial_y \tilde u [z_{xx} -cz +   2\tilde u z + z^2  ]  \nonu\\
& = & - \rho'_1(t) \int_\R \tilde u_x   z^2 + O(\ve e^{-\ve\ga |\rho(t)|} \|z(t)\|^2_{L^2(\R)}).
\eea
Similarly, we have from (\ref{defz})
\bea\label{a02}
(\ref{Fp2b})  & = &  c_1' (t) \int_\R \partial_c \tilde u [ z_{xx} -cz +  2\tilde u z + z^2  ]  =   c_1' (t) \int_\R  \partial_c \tilde u   z^2 \nonu\\
& &  \qquad + O(\ve e^{-\ve\ga |\rho(t)|} \|z(t)\|^2_{L^2(\R)}   + \ve^{1/2} |c_1'| e^{-\ve\ga |\rho(t)|} \|z(t)\|_{L^2(\R)}).
\eea
On the one hand, we have from (\ref{SSS}) and (\ref{TRANS4}),
\bea\label{a03}
& & \abs{  \int_\R \tilde S[\tilde u] \big\{ z_{xx} -cz + 2\tilde u z + z^2  \big\}  }  \lss  \nonu\\
& & \qquad \lss \abs{  \int_\R\partial_x \tilde S[\tilde u] z_{x} } +  (1+K^* \ve^{1/2})\abs{\int_\R \tilde S[\tilde u] z}  +  \abs{\int_\R \tilde S[\tilde u] \tilde u z } \nonu \\
& & \qquad \lss   \abs{  \int_\R\partial_x \tilde S[\tilde u] z_{x} } +   (1+K^* \ve^{1/2})\abs{\int_\R \tilde S[\tilde u] z} \nonu\\
& & \qquad  \lss K^*\ve^{2}e^{-\ve\ga|\rho(t)|} + (K^*)^2 \ve^{3/2} e^{-\ve\ga|\rho(t)|} \Big[\int_\R e^{-\ga_0|y|} z^2\Big]^{1/2} .
\eea
Concerning the second and third terms in (\ref{Fp3b}), 
\[
\frac 12 c_1'(t) \int_\R z^2 +  \frac 12 \ve f_1(t) \int_\R z^2  \lesssim  (|c_1'(t)| +\ve e^{-\ve\ga|\rho(t)|}) \|z(t)\|_{L^2(\R)}^2.
\]
Finally, 
\bea\label{a04}
(\ref{Fp4})  & =&       - \int_\R  (\tilde u_t + \rho' \tilde u_x  - c' \partial_c \tilde u) z^2 +  \rho'\int_\R  \tilde u_x z^2 \nonu\\
& &  -c'\int_\R  \partial_c \tilde u z^2 +O(\ve e^{-\ve\ga|\rho(t)|} \|z(t)\|_{L^2(\R)}^2).
\eea
We obtain
\bee
\mathcal F'(t) & \lesssim  & |c_1'(t)| \|z(t)\|_{L^2(\R)}^2 +  \ve e^{-\ga\ve |\rho(t)|} \|z(t)\|^2_{L^2(\R)} + \ve\|z(t)\|_{H^1(\R)}^3 + K^*\ve^{2}e^{-\ve\ga|\rho(t)|} \\
& & \qquad  + \ (K^*)^2 \ve^{3/2} e^{-\ve\ga|\rho(t)|} \Big[\int_\R e^{-\ga_0|y|} z^2\Big]^{1/2}.
\eee
Using (\ref{intc1}), we finally get after integration in time (here we use the condition $t\leq T^* \lss \ve^{-1-2\delta_0}$)
\bee
\mathcal F(t) -\mathcal F(0) & \lesssim & (K^*)^3 \ve^{\frac 32-\delta_0}  + (K^* \ve)^2 +  \int_{0}^t  \ve e^{-\ga\ve |\rho(s)|} \|z(s)\|_{H^1(\R)}^2 ds \\
& & \qquad + \ (K^*)^2 \ve^{3/2}  \int_0^t  e^{-\ve\ga|\rho(s)|} \Big[\int_\R e^{-\ga_0|y|} z^2\Big]^{1/2}ds.
\eee
Note that thanks to the Cauchy-Schwarz inequality and Lemma \ref{VL}, we have 
\[
(K^*)^2 \ve^{3/2}  \int_0^t  e^{-\ve\ga|\rho(s)|} \Big[\int_\R e^{-\ga_0|y|} z^2\Big]^{1/2}ds \lss (K^*)^2 \ve^{3/2},
\]
for $\ve$ small (depending on $K^*$). Indeed, we just need to justify that $\abs{\int_{0}^{t} \ve e^{-\ga\ve |\rho(s)|} ds} \lss 1$, independent of $\ve$ and $K^*$. It is not difficult to see that the estimate above  holds since $\rho'(s) \geq  \frac{9}{10}c(s) \geq \frac{8}{10} \min \{c_f,1\}>0 $.   Therefore
\[
\mathcal F(t) -\mathcal F(0)  \lesssim  (K^*)^3 \ve^{\frac 32-\delta_0}  + (K^* \ve)^2 +  \int_{0}^t  \ve e^{-\ga\ve |\rho(s)|} \|z(s)\|_{H^1(\R)}^2 ds, 
\]
as desired. The cases $p=3$ and $4$ are similar.
\end{proof}
We are finally in position to improve (\ref{Tstar}). Indeed, since from Lemma \ref{DEFZ}, $\mathcal F(0) \lss  \ve,$ using (\ref{Coer2}) and Lemma (\ref{IntF}) we get
\bee
\|z(t)\|_{L^2(\R)}^2   & \lss &   \ve +  (K^*)^4 \ve^{2-\delta_0} + (K^*)^3 \ve^{\frac 32-\delta_0}  + (K^*)^2\ve^2 \\
& &   +  \int_{0}^t  \ve e^{-\ga\ve |\rho(s)|} \|z(s)\|_{H^1(\R)}^2 ds.
\eee
Now, by Gronwall's inequality (see e.g. \cite{Mu2} for a detailed proof),
\be\label{RMCF}
\|z(t)\|_{H^1(\R)}^2 \lss \ve  + (K^*)^3 \ve^{\frac 32-\delta_0}.
\ee
with constant independent of $K^*$ and $\ve$. 

\medskip

Let us come back to the main proof. From estimate (\ref{RMCF}), and taking $\ve$ small, and $K^*$ large enough, we obtain that for all $t\in [0, T^*]$,
\be\label{KKa}
\|z(t)\|_{H^1(\R)}^2 \leq  \frac 14 (K^*)^2 \ve.
\ee
Therefore, we improve the estimate on $z(t)$ (\ref{TRANS4}), and therefore (\ref{Tstar})-(\ref{Tstar1}) are also improved. The proof of Proposition \ref{prop:I} is complete.
\end{proof}

\bigskip

\section{Proof of the Main Theorems}\label{5}

\medskip

We are now in position to give a direct proof of Theorem \ref{MT}. The proof is very similar to the corresponding proof of Lemma \ref{ODE0}. Indeed, we have (\ref{INT41}) for all time $t\in [0, T]$; in particular, at $t=T$ one has
\be\label{Conc1}
\|u(T)  -\tilde u(T, c(T), \rho(T)) \|_{H^1(\R)} \lss \sqrt{\ve}.
\ee
Note that, from (\ref{c1}) and (\ref{rho1}), since $c(t)>c_m$ by (\ref{r1}),
\bee
c^{-\frac1{p-1}}(t) c'(t) & =&  -\ve \la_p a_0'(\ve \rho(t)) c(t)  + c^{-\frac1{p-1}}(t) c_1'(t) \\
& =& -\ve \la_p a_0'(\ve \rho(t)) \rho'(t) +  \ve^2 \la_p a_0'(\ve \rho(t)) f_2(t)  \\
& & + \ve \la_p a_0'(\ve \rho(t)) \rho_1'(t) + c^{-\frac1{p-1}}(t) c_1'(t).
\eee
Hence, if $p=2$, and using (\ref{intc1}),
\[
\log c(t) -\log c(0) = -\la_2 [a_0(\ve \rho(t)) - a_0(\ve \rho(0))]  + O(\sqrt{\ve}),
\]
from which we obtain
\be\label{C2}
c(t) = e^{-\la_2 a_0(\ve \rho(t))}(1+O(\sqrt{\ve})), \quad p=2.
\ee
Similarly, if $p=3$ or $4$,
\be\label{Cp}
c(t) = \Big[  1-  \la_p \frac{(p-2)}{p-1} a_0(\ve \rho(t)) \Big]^{\frac{p-1}{p-2}} (1+O(\sqrt{\ve})).
\ee
Now, we perform a detailed asymptotic analysis of $(c(t), \rho(t))$. First of all, note that from (\ref{C2})-(\ref{Cp}), $c(t)$ is strictly positive for all time. Indeed, the case $p=2$ is direct, and for $p=3$ or $p=4$, we have $c(t) \geq 1 -K\sqrt{\ve}$ in the case where $a_\infty$ is negative (or $c_f>1$, see (\ref{ainf}) and (\ref{ahyp})), and
\[
c(t) \geq \Big[  1-  \la_p \frac{(p-2)}{p-1} \|a_0\|_\infty \Big]^{\frac{p-1}{p-2}}  - K \sqrt{\ve},
\]
for the case $a_\infty>0$ ($c_f<1$). From (\ref{Linf}) we have $\|a_0\|_{\infty} = a_\infty= \frac{(p-1)}{\la_p(p-2)}(1-c_f^{\frac{p-2}{p-1}})$. Therefore,
\[
c(t) \geq c_f - K\sqrt{\ve}>0.
\]
We conclude that, for $\ve>0$ small, $c(t) \geq \frac{99}{100}\min\{ 1,c_f\}>0.$ Hence $\rho(t)$ is increasing and $\rho(t) -\rho(0) \geq  \frac{99}{100} \min\{1,c_f\} t $,
which implies that $ \rho(T) \gtrsim T$. Moreover, from (\ref{ahyp}), $a(\ve \rho(T)) = a_\infty + O(\ve^{10})$. Taking $t=T$ in (\ref{C2}) and (\ref{Cp}), we have
\[
c(T) = e^{ -\la_2 a_\infty} (1+O(\sqrt{\ve})) = c_f (1+O(\sqrt{\ve})), \quad p=2,
\]
and
\[
c(T) = \Big[  1-  \la_p \frac{(p-2)}{p-1} a_\infty \Big]^{\frac{p-1}{p-2}} (1+O(\sqrt{\ve}))= c_f(1+O(\sqrt{\ve})), \quad p=3,4,
\]
as desired. Finally, from (\ref{defW}), one has
\[
\| \tilde u(T, c(T), \rho(T))  - Q_{c_f} (\cdot - \rho(T)) \|_{H^1(\R)} \lss \sqrt{\ve}.
\] 
Using (\ref{Conc1}) and the triangle inequality, the first estimate in (\ref{MT1}) follows. Concerning the second one, it is a consequence of (\ref{rho1}).

\medskip

Finally, Corollary \ref{Cor2} is just a consequence of the behavior of $c(t)$ in (\ref{C2}), (\ref{Cp}) and (\ref{ahyp}).

\end{document}